\documentclass{amsart}
\usepackage{amssymb}
\usepackage{mathtools}
 \usepackage{graphicx}\usepackage{amsfonts}
\usepackage{pstricks}
\usepackage{pst-node}
\usepackage{pstricks-add}
\usepackage{pst-plot}
\usepackage[all,poly,2cell,ps]{xy}

\def\R{{\mathbb R}}

\def\N{{\mathbb N}}

\def\<{\langle}
\def\>{\rangle}
\def\P{{\mathbb P}}
\def\E{{\mathbb E}}
\def\eps{\epsilon}
\def\I{\mathcal I}
\def\i{\underline i}
\def\j{\underline j}
\def\r{\underline r}
\def\ep{\underline{\eps}}
\def\k{\underline k}
\def\x{\underline{x}}

\def\m{\underline{m}}
\def\n{\underline{n}}

\def\p{\mathfrak{p}}
\def\pp{\mathfrak{P}}
\def\cP{\mathcal P}
\def\cc{\mathcal C}
\def\ss{\mathcal S}
\def\qq{\mathfrak{Q}}

\def\c{\mathfrak{c}}
\def\d{\mathfrak{d}}
\def\q{\mathfrak{q}}
\def\cC{\mathcal{C}}
\newcommand{\bel}{\begin{equation}\label}
\newcommand{\ee}{\end{equation}}
      \newtheorem{theorem}{Theorem}[section]
       
       \newtheorem{corollary}[theorem]{Corollary}
       \newtheorem{lemma}[theorem]{Lemma}
       
\theoremstyle{definition}
\newtheorem{definition}{Definition}[section]

\newtheorem{example}{Example}[section]
\begin{document}
\title[A new prior for discrete DAG models]{A new prior for discrete DAG models with  a restricted set of directions}
\author{H\'el\`ene Massam}
\address{Department of Mathematics and Statistics, York University, Toronto, M3J1P3, Canada}
\email{massamh@yorku.ca}
\author{Jacek Weso\l owski}
\address{Wydzial Matematyki i Nauk Informacyjnych, Politechnika Warszawska, Warszawa, 00662, Poland}
\email{wesolo@mini.pw.edu.pl}


\begin{abstract}
In this paper, we first develop a new family of conjugate prior distributions for the cell parameters of discrete graphical models Markov with respect to a set $\cP$ of moral directed acyclic graphs with skeleton a given decomposable graph $G$. Such families arise when the set of conditional independences between discrete variables is given and can be represented by a decomposable graph and additionally, the direction of certain edges is imposed by the practitioner. This family, which we call the $\cP$-Dirichlet, is a generalization of the hyper Dirichlet given in Dawid and Lauritzen (1993): it keeps the  strong directed hyper Markov property for every DAG in $\cP$ but increases the flexibility in the choice of its parameters, i.e. the hyper parameters.

Our second contribution is a characterization of the $\cP$-Dirichlet, which yields, as a corollary, a characterization of the hyper Dirichlet and a characterization of the Dirichlet also. Like that  given by Geiger and Heckerman (1997), our characterization of the Dirichlet is based on local and global independence of the probability parameters but we need not make the assumption of the existence of a positive density function. We use the method of moments for our proofs.
\end{abstract}

\maketitle

\section{Introduction}
The Dirichlet distribution and distributions derived from the Dirichlet are essential ingredients of Bayesian inference in the analysis of discrete data. For high-dimensional data, Dirichlet-type distributions are often used in conjunction with graphical  models. Let $V=\{1,\ldots,d\}$ be a finite set indexing the variables. A graphical model for the discrete random variable $X=(X_v, v\in V)$ is a statistical model where the dependences between $X_v, v\in V$ are expressed by means of a graph $G$.
We will assume here that the data is gathered under the form of a $d$-dimensional contingency table and that the cell counts follow a multinomial  distribution with cell probabilities $(p(\i), \i\in \I)$ where $\I$ is the set of cells in the contingency table.

If there are no independences between the variables $X_v, v\in V$, then the cell counts follow a standard multinomial distribution and  the Diaconis-Ylvisaker (1979) prior distribution on
$(p(\i), \i\in \I)$ is the Dirichlet distribution. If the conditional independences between the variables can be expressed by means of a directed acyclic graph (henceforth abbreviated DAG),  the usual priors are independent Dirichlet prior distributions on
\begin{equation}
\label{B}
\Big(p(i_v|X_{\p_v}=\i_{\p_v}), i_v\in \I_v\Big), \i_{\p_v}\in \I_{\p_v}, v\in V,
\end{equation}
where, for $A\subset V$, $\I_A$ is the finite set of values that $X_A=(X_v, v\in A)$ can take, $\p_v$ denotes the set of parents of vertex $v$ in the DAG and $\i_{\p_v}=(i_l,l\in \p_v)$ denotes the $\p_v$-marginal cell. This means that the probability parameters associated with each node are mutually independent (global independence) and for a given node, the parameters associated with various instances of its parents are also mutually independent (local parameter independence).

Geiger and Heckerman (1997, Theorem 3) have shown that if the distribution of positive random variables that sum to 1 has a strictly positive density and possesses the property of local and global independence for the two directions $\{1,2,\ldots, d-1,d\}$ and $\{d,1,2,\ldots,d-1\}$ of a complete DAG on $V$, this distribution must be the Dirichlet distribution. This is a characterization of the Dirichlet through local and global independence using two particular complete DAG's. This result has important practical ramifications for the choice of the hyper parameters of the Dirichlet priors for model selection, as emphasized by Geiger and Heckerman (1997). Indeed, for model selection in the class of models Markov with respect to DAG's on $V$, if we request that  a node with the same parents in two distinct DAG's has identical prior distribution on the parameters associated with this node in both structures (this is called parameter modularity) and if we also
request local and global independence of the parameters for all possible DAG's, then the priors on  the variables in \eqref{B} must all be derived from one single Dirichlet distribution $\mbox{Dir}(\alpha(i), i\in I)$ on $(p(\i),\,\i\in \I)$. The hyper parameters $\alpha(i)$ of such a Dirichlet are of the form
\begin{equation}
\label{1}
\alpha(i)=\alpha \theta(i)
\end{equation}
where $\theta(i)$ is the expected value of $p(i)$ and $\alpha$,  called the {\em equivalent sample size}, (which we can think of as the total cell count in a fictive contingency table representing prior knowledge) will be added to the actual total cell count $N=\sum_{\i\in \I}n(\i)$ in posterior inference. This $\alpha$ will therefore
represent our confidence in our choice of the particular fictive table. Most of the time, in practice and in the absence of prior expert information, we take $\theta(i)=\frac{1}{|\I|}$ where $|\I|$ is the total number of cells. This implies that, often in practice, our flexibility in the choice of the hyper parameters is restricted to the choice of $\alpha$.
\vspace{2mm}

If  there are  conditional independences between the variables  and they can be expressed by means of a decomposable undirected graph $G$, then the multinomial distribution on the cell counts is Markov with respect to $G$ and the Diaconis-Ylvisaker (1979) prior  on $(p(\i), \i\in \I)$ is the hyper Dirichlet defined by David and Lauritzen (1993). A simple calculation (see equations \eqref{cands} in Section \ref{prelim}) shows that the parameters of the hyper Dirichlet  have an interpretation of the type
 \eqref{1} and we are therefore
faced with the same lack of flexibility in the choice of the hyper parameters: again, we are restricted to the choice of one parameter only. Moreover it is intuitively clear (and it can be proved, see Theorem \ref{hyp-dir})
 that if, for any DAG with skeleton $G$, Markov equivalent to $G$, we make the change of variable from $p(i),i\in \I$ to the variables in equation \eqref{B}, i.e., $\Big(p(i_v|X_{\p_v}=\i_{\p_v}), i_v\in \I_v\Big), \i_{\p_v}\in \I_{\p_v}, v\in V$, then, for each of these DAGs, the  distribution induced from the hyper Dirichlet also possesses the property of local and global independence. One can then be led to think that the property of local and global independence in all possible directions imposes strict restrictions on the choice of the hyper parameters of the prior on  $p(i),i\in \I$.

 \vspace{3mm}

In this paper,  we therefore consider a family $\cP$ of DAG's with skeleton a decomposable graph $G$ and Markov equivalent to $G$ but we do not include in this family all the DAG's Markov equivalent to $G$: $\cP$ is a restricted family of such DAG's.
For the set of discrete multinomial models Markov with respect to any of the DAG's in $\cP$, we  build a new family of conjugate prior distributions for the parameters $({\bf p}(\i),\, \i \in \I)$ of this set of models. As we show in Section 5.3, such a family of priors will be the right family of prior distributions to use  if we are given by an expert the set of conditional independences (those given by a decomposable graph $G$) and  the restriction  that some edges must have a given direction. In that case, there is no need to require local and global independence for all DAG's Markov with respect to $G$. We only need to require it for those DAGs' containing the compulsory edges.
This new family of priors possesses the property of local and global independence and also, as we shall see later, the strong directed hyper Markov property. Using these priors thus facilitates posterior inference. Moreover, because of the restriction on $\cP$, it offers more hyper parameter flexibility than the hyper Dirichlet. The development and study of this new family of priors, called $\cP$-Dirichlet, is the first contribution of this paper. The second contribution is  a characterization (Theorem \ref{main1}) of this new family through local and global independence. As corollaries, we obtain a characterization of the hyper-Dirichlet and also the characterization of the Dirichlet as given by Geiger and Heckerman but without assuming the existence of a density for ${\bf p}=({\bf p}(\i),\, \i\in \I)$.

The remainder of this paper is organized as follows. The  $\cP$-Dirichlet family of priors on $(p(\i), \i\in \I$,  is defined   in Section \ref{def}, through  local and global independence with respect to the collection $\cP$ of DAGs.
In Section \ref{moments}, we first give an example of a $\cP$-Dirichlet family and then we  derive a general formula for the moments of the $\cP$-Dirichlet which will be used in its characterization,  Theorem \ref{main1}. In the expression of the moments, we will see that the role played by the collection of cliques and separators in the hyper Dirichlet is played by two larger collections, $\qq$ and $\pp$, of complete subsets of the decomposable graph $G$ which include respectively, the set of cliques and separators. These additional complete subsets yield more parameters for the $\cP$-Dirichlet distribution and thus increase flexibility.  In Section \ref{prior}, we give the dimension of the $\cP$-Dirichlet  family and show that it is always larger than (or equal to )  the dimension of the hyper Dirichlet family with the same skeleton. We also show that the $\cP$-Dirichlet is a conjugate family for the multinomial distribution Markov with respect to $G$ and that it has the strong directed hyper Markov for any DAG in $\cP$.
In Section \ref{characterization}, we show that, when $\cP$ is what we  call a   {\em separating} family of DAG's,  local and global independence characterizes the $\cP$-Dirichlet distribution. We will look at special cases. When  $G$ is a tree, for the $\cP$-Dirichlet to be the hyper Dirichlet, $\cP$  needs only be the set of  DAG's with their root at a leaf of the tree, not the set of all possible DAG's Markov equivalent to $G$ . When $G$ is complete, a separating  family of DAG's can be formed by taking the orders $\{1,2,\ldots,d\}$ and $\{d,1,2,\ldots,d-1\}$ as was done by Geiger and Heckerman (1997) but this is not the only separating family and the characterization of the Dirichlet can be obtained through any separating family of DAG's on a complete graph $G$.

We conclude with Section \ref{apps} summarizing the contributions of this paper.
Proofs of the moment formula in Theorem \ref{momenty} and the characterization, Theorem \ref{main1} are given in the Appendix, Section \ref{proofs}. The next section is devoted to preliminaries.

\section{Preliminaries}
\label{prelim}
\subsection{ Graph theoretical notions}
Let $V=\{1,\ldots,d\}$ be a finite set of indices for the  $d$ criteria defining the contingency table with our data. We assume that the criterion labelled by $v\in V$ can take values in a finite set $I_v$. Let
$$\I=\prod_{v\in V}\I_{v}$$
be the  set of cells $\i=(i_v,\;v\in V)$. If $D\subset V$  and $\i\in \I$ we write
$\i_D=(i_v, {v\in D})$ for the $D$-marginal cell. Let $G=(V,E)$ be an undirected  graph as defined above. A graph is said to be decomposable if it does not have any cycle of length greater than or equal to 4. A subset $D$ of $V$ is said to be complete if all vertices of $D$ are linked to each other with an edge. A clique is a  complete subset maximal with respect to inclusion. Let $\cC$ denote the set of cliques. For any given order $C_1,\ldots,C_K$ of the cliques of $G$ we will define
$$H_j=\cup_{l=1}^jC_l, \;\;S_j=C_j\cap H_{j-1}\;\;\mbox{and}\;\;R_j=C_j\setminus H_{j-1}=C_j\setminus S_j, \;\;\;j=1,\ldots K,$$
with $S_1=\emptyset$, called respectively the $j$-th history, $j$-th separator and $j$-th residual.
An order $C_1,\ldots,C_K$ of the cliques of $G$ is said to be perfect if for any $j>1$, there exists an $i<j$ such that
$$(\cup_{l=1}^{j-1} C_l)\cap C_j\subset C_i.$$
 It can be shown (Proposition 2.17,  Lauritzen, 1996) that a graph $G$ is decomposable if and only if its set of cliques admits a perfect order. It can also be shown that the set of separators $\ss$ associated with a perfect ordering of the cliques is independent of the perfect ordering considered and it is the set of minimal separators for the graph $G$.

For a given decomposable  $G$, we will consider the DAG's with skeleton $G$ which are moral. A DAG is said to be moral
if when $v$ and $v'$ are both parents of $w$, then there is an edge between $v$ and $v'$. In this paper, we will deal exclusively with moral DAG's.
  Note than any  DAG with skeleton $G$ is in one-to-one correspondence with the function $\p:V\to 2^V$ which describes the parents of each vertex or with the function $\c:V\to 2^V$ which describes the children of each vertex. We will write $\p_v$ and $\c_v$ for the set of parents and children of $v\in V$, respectively. The set of descendants, that is children of children and so on, of a vertex $v$ is denoted ${\mathfrak d_v}$ and the set of nondescendants will be denoted ${\mathfrak n\mathfrak d}_v$.
 Additionally denote $\q_v=\p_v\cup\{v\}$ for any $v\in V$ and let $\q:V\to 2^V$ denote the corresponding mapping.  Note that for any $v\in S$ for some separator $S$, due to morality, there exists a unique clique $C\supset S$ such that $\p_v\subset C$.

 We will now define some additional graph-theoretical notions needed in this work.
 \begin{definition}
For $S\in\mathcal S$ and $C\in\mathcal C$, we say that $C$ and $S$ {\em are paired} by a given perfect ordering  $o=(C_1,\ldots,C_K)$ of the cliques if $$\exists\,l\in\{1,\ldots,K\}:\qquad S=S_l,\quad C=C_l$$ (here we allow that $S=S_1=\emptyset$). If $S$ and $C$ are paired by $o$, we write $S\stackrel{o}{\to} C$.
\end{definition}
 \begin{definition}
 Given a parent function $\p$, a perfect order $o =(C_1,\ldots, C_K)$  of the cliques is said to be $\p$-perfect if for any $l=1,\ldots,K$ there exists a $v\in C_l\setminus S_l$ such that $S_l=\p_v$.
 \end{definition}
With a given  $\p\in\cP$ and a $\p$-perfect  order  $o=(C_1,\ldots,C_K)$ of the cliques,  we associate the following numbering of the vertices: $\forall\,l\in\{1,\ldots,K\}$
\bel{numb}
v=\left\{\begin{array}{ll}
           v_{l,s_l+1} & \mbox{iff}\;\;\p_v=S_l,\\
           v_{l,j} & \mbox{iff}\;\;\p_v=\q_{v_{l,j-1}},\;j=s_l+2,\ldots,c_l,
         \end{array}\right.
\ee
where  $c_l$ and $s_l$ denote the cardinality of $C_l$ and $S_l$ respectively.
We note that if $v=v_{l,c_l}$ then $\q_v=C_l$, $l=1,\ldots,K$ and  if $v=v_{1,1}$, the source vertex of $\p$, then $\p_v=\emptyset$.



\subsection{Markov properties and the hyper Dirichlet}
Let ${\bf X}=(X_1,\ldots,X_d)$ be a discrete random vector with variables  $X_v$ taking values in $\I_v$, $v\in V$. Given  an undirected graph $G$,
the distribution of $X$ is said to be Markov with respect to the  $G$ if $X_i$ and $X_j$ are independent  given ${\bf X}_{V\setminus \{i,j\}}$ whenever the pair $\{i,j\}$ does not belong to $E$.

\noindent Given a DAG ${\mathcal G}$ ,  we say that ${\bf X}$ is Markov with respect to ${\mathcal G}$ if, for any $v\in V$,
\bel{dagind}
X_v\perp {\bf X}_{\mathfrak{n} \mathfrak{d}_v}\;|\;{\bf X}_{\p_v}.
\ee
Therefore if ${\bf X}$ is  Markov with respect to ${\mathcal G}$, we have
\bel{dag}
p(\i):=\P({\bf X}=\i)=\P(X_v=i_v,\;v\in V)=\prod_{v\in V} \P(X_v=i_v|{\bf X}_{\p_v}=\i_{\p_v}).
\ee
The two formulations (\ref{dagind}) and (\ref{dag}) are equivalent.

In this paper, we consider only  DAGs ${\mathcal G}$  with a decomposable skeleton $G$ and which  encode the same conditional independences as $G$. By Lemma 3.21 of Lauritzen (1996), such a DAG is necessarily moral and it has been shown in Andersson, Madigan, Perlman and Triggs (1997) that, for this  DAG, we can use a numbering of the  vertices  of the type \eqref{numb}. With this numbering
the two models,  Markov with respect to $G$ and  ${\mathcal G}$,  define the same set of conditional independences: they are said to be Markov equivalent.
It is well known that for any $\i\in\I$, we have the following decomposition of the probability that ${\bf X}=\i=(i_1,\ldots,i_d)$,
\bel{dec}
\P({\bf X}=\i)=\P(X_v=i_v,\;v\in V)=\frac{\prod_{l=1}^K\,\P({\bf X}_{C_l}=\i_{C_l})}{\prod_{l=2}^K\,\P({\bf X}_{S_l}=\i_{S_l})}.
\ee

Then, if the cell counts $(N(\i), \;\i\in \I)$  follow a  Multinomial $(N, p(\i), \i\in \I)$  with $p(\i)$  as given in (\ref{dag}), the density of $(N(\i), \i\in \I)$ with respect to the counting measure is proportional to
\begin{eqnarray}
\prod_{\i\in \I}p(\i)^{n(\i)}
&=&\frac{\prod_{l=1}^K\prod_{\i_{C_l}\in  \I_{C_l}} \P({\bf X}_{C_l}=\i_{C_l})^{n(\i_{C_l})}}
{\prod_{l=2}^K\prod_{\i_{S_l}\in  \I_{S_l}}\P({\bf X}_{S_l}=\i_{S_l})^{n(\i_{S_l})}}\nonumber\\
&=&\prod_{v\in V} \prod_{i_{\p_v}\in \I_{\p_v}}\prod_{i_v\in \I_v}
\P({\bf X}_v=i_v|{\bf X}_{\p_v}=\i_{\p_v})^{n(\i_{\q_v})}\label{multinomial}
\end{eqnarray}
where the two equalities follow from  (\ref{dec}) and (\ref{dag}), respectively.

We note that for any $l=1,\ldots,K$ the marginal counts satisfy
\bel{cdec}
\sum_{\i_{C_l}\in \I_{C_l}}n(\i_{C_l})=N,\;\sum_{\j_{C_l}\in \I_{C_l}:\j_{S_l}=\i_{S_l}}\,n(\j_{C_l})=n(\i_{S_l}),\;\sum_{\i_{S_l}\in \I_{S_l}}n(\i_{S_l})=N
\end{equation}
and, for $v\in C_l$ and $\q_v\subseteq C_l$,
\begin{equation}
\label{cdag}
\sum_{\j_{C_l}\in I_{C_l}:\j_{\q_v}=\i_{\q_v}}n(\j_{C_l})=n(\i_{\q_v}), \;i_{\q_v}\in \I_{\q_v}.
\ee

In Bayesian inference we view the parameter of the multinomial distribution as a random vector.
Let ${\bf p}=({\bf p}(\i),\,\i\in \I)$ be the vector of random cell probabilities, that is ${\bf p}(\i)>0$ for any $\i\in \I$ and $\sum_{\i\in\I}\,{\bf p}(\i)=1$. Note that we write ${\bf p}=({\bf p}(\i),\,\i\in\I)$ for the random vector of cell probabilities, and $p=(p(\i),\,\i\in\I)$ for its value. Let
$$
\P_p({\bf X}=\i):=\P({\bf X}=\i|{\bf p}=p)=p(\i),\quad \i\in \I.
$$
The random variables $\P_{\bf p}({\bf X}=\i)=\P({\bf X}=\i|{\bf p})$, $\i\in\I$, are the variables of interest in this paper.

As mentioned in the introduction the Diaconis-Ylvisaker conjugate  prior on ${\bf p}$ is the hyper Dirichlet first identified by Dawid and Lauritzen (1993) with density
\begin{equation}
\label{hp}
\frac{\Gamma(\alpha)\prod_{l=2}^K\prod_{\i_{S_l}\in \I_{S_l}} \Gamma(\alpha^{S_l}_{\i_{S_l}})}
{\prod_{l=1}^K\prod_{i_{C_l}\in \I_{C_l}} \Gamma(\alpha^{C_l}_{\i_{C_l}})}\,\frac{\prod_{l=1}^K
\prod_{\i_{C_l}\in  \I_{C_l}}\,(p^{C_l}_{\i_{C_l}})^{\alpha^{C_l}_{\i_{C_l}}-1}}
{\prod_{l=2}^K
\prod_{\i_{S_l}\in  \I_{S_l}}\,(p^{S_l}_{\i_{S_l}})^{\alpha^{S_l}_{\i_{S_l}}-1}} ,
\end{equation}
where $p^{C_l}_{\i_{C_l}}$ and $p^{S_l}_{\i_{S_l}}$ are the values of the random variables ${\bf p}^{C_l}_{\i_{C_l}}=\P_{\bf p}({\bf X}_{C_l}=\i_{C_l})$ and ${\bf p}^{S_l}_{\i_{S_l}}=\P_{\bf p}({\bf X}_{S_l}=\i_{S_l})$, $l=1,\ldots,K$, respectively, and where the hyperparameters satisfy constraints parallel to \eqref{cdec}, namely
\bel{cdecalpha}
\sum_{\i_{C_l}\in \I_{C_l}}\alpha^{C_l}_{\i_{C_l}}=\alpha,\;\sum_{\j_{C_l}\in \I_{C_l}:\j_{S_l}=\i_{S_l}}\,\alpha^{C_l}_{\j_{C_l}}=\alpha^{S_l}_{\i_{S_l}},\;\sum_{\i_{S_l}\in \I_{S_l}}\,\alpha^{S_l}_{\i_{S_l}}=\alpha.
\ee
Unlike the case of the Dirichlet on the complete graph, this density is typically defined on a complicated manifold determined not only by summation to one but also by all the conditional independence properties encoded in $G$. This is one of the reasons our approach is through moments and not densities.

We now want to recall the expression of the moments of the hyper Dirichlet.
For simplicity, let us write
$$\alpha(G,\I)=(\alpha,\,\alpha^{C_l}_{\k},\,\k\in\I_{C_l},\,l=1,\ldots,K,\,\alpha^{S_l}_{\k},\,\k\in\I_{S_l},\,l=2,\ldots,K).$$
For any $d$-way table $\r=(r_{\i},\,\i\in\I)$ of non-negative integers, we write the $E$-marginal counts as
$$r^E_{\underline{e}}=\sum_{\i:\,\i_E=\underline{e}}r(\i).$$
In particular, we will use $\r^{C_l}_{\i_{C_l}}$, $\r^{S_l}_{\i_{S_l}}$. By analogy with $\alpha(G,\I)$, we will also use the notation $r(G,\I)$ and write the normalizing constant of the hyper Dirichlet as
\bel{nc}
\mathcal{Z}(\alpha(G,\,\I))=\frac
{\prod_{l=1}^K\prod_{\i_{C_l}\in \I_{C_l}} \Gamma(\alpha^{C_l}_{\i_{C_l}})}{\Gamma(\alpha)\prod_{l=2}^K\prod_{\i_{S_l}\in \I_{S_l}} \Gamma(\alpha^{S_l}_{\i_{S_l}})}
\ee

The moments are then equal to
\bel{moments}
\E\,\prod_{\i\in\I}\,\left[\P_{\bf p}({\bf X}=\i)\right]^{r_{\i}}=\E\, \frac{\prod_{l=1}^K\prod_{\i_{C_l}\in  \I_{C_l}} \P_{\bf p}({\bf X}_{C_l}=\i_{C_l})^{r^{C_l}_{\i_{C_l}}}}
{\prod_{l=2}^K\prod_{\i_{S_l}\in  \I_{S_l}}\,\P_{\bf p}({\bf X}_{S_l}=\i_{S_l})^{r^{S_l}_{\i_{S_l}}}}=\frac{\mathcal{Z}(\alpha(G,\I)+r(G,\I))}{\mathcal{Z}(\alpha(G,\I))}.
\ee
In particular,
\begin{eqnarray*}
\E\,\left[\P_{\bf p}({\bf X}=\i)\right]=
\frac{\prod_{l=1}^K \Gamma(\alpha^{C_l}_{\i_{C_l}}+1)
}{
\Gamma(\alpha+1)\prod_{l=2}^K \Gamma(\alpha^{S_l}_{\i_{S_l}}+1)
}
\frac{
\Gamma(\alpha)\prod_{l=2}^K \Gamma(\alpha^{S_l}_{\i_{S_l}})
}{\prod_{l=1}^K \Gamma(\alpha^{C_l}_{\i_{C_l}})
}=
\frac{\prod_{l=1}^K \alpha^{C_l}_{\i_{C_l}}}{\alpha \prod_{l=2}^K \alpha^{S_l}_{\i_{S_l}}}
\end{eqnarray*}
Together with the constraints \eqref{cdecalpha}, this shows that we can write
\begin{eqnarray} \label{cands}
\alpha^{C_l}_{\i_{C_l}}=\alpha\theta (i_C),\;\;\alpha^{S_l}_{\i_{S_l}}=\alpha\theta (i_S),
\end{eqnarray}
where $\theta(i_C)=\E\,\left[p(i_C)\right]$ and $\theta(i_S)=\E\,\left[p(i_S)\right]$, a relationship similar to \eqref{1} as mentioned in the introduction.
\vspace{2mm}

\section{$\p$-Dirichlet and $\cP$-Dirichlet distributions}
\label{def}
Let $G=(V,E)$ be a decomposable graph. Let ${\bf X}=(X_1,\ldots,X_d)$, $d\ge 2$, be a random vector which is Markov with respect to $G$ and assumes values in the set $\I=\I_1\times\ldots\times \I_d$, where $\I_j$ is a discrete set, with distinguished zero state $\emptyset_j$, $j=1,2,\ldots,d$.



Since a moral DAG Markov equivalent to $G$ is in 1-1 correspondence with the parent function
$\p: V\rightarrow 2^V$
which describes the parents of each vertex $v\in V$, following \eqref{dag}, we have that
$$\P_{\bf p}({\bf X}=\i)=\prod_{v\in V}\P_{\bf p}(X_v=i_v|{\bf X}_{\p_v}=i_{\p_v}).$$
Letting $\q_v=\p_v\cup\{v\}$, $v\in V$, we can thus define the random variables

$$
{\bf p}^{v|\p_v}_{m|\k}:=\P_{\bf p}(X_v=m|{\bf X}_{\p_v}=\k)=\tfrac{{\bf p}^{\q_v}((\k,m))}{{\bf p}^{\p_v}(\k)},\qquad m\in\I_v,\;\k\in\I_{\p_v},\quad v\in V,
$$

\vspace{2mm}\noindent
where ${\bf p}^D(\n)=\sum_{\j\in\I:\,\j_D=\n}\,{\bf p}(\j)$ for $\n\in\I_D$, $D\subset V$ and we have

\bel{rep}
{\bf p}(\i)=\prod_{v\in V}\,{\bf p}^{v|\p_v}_{i_v|\i_{\p_v}},\qquad \i\in \I.
\ee
We will say that a  random vector $({\bf p}(\i))_{\i\in\I}$ is associated with the graph $G$ if it factorizes with respect to $G$, i.e. $p(i)$ can be written as
$$p(i)=\frac{\prod_{C\in {\mathcal C}}p(i_C)}{\prod_{S\in {\mathcal S}}p(i_S)},\;\;i\in \I,$$
that is, if ${\bf X}=(X_1,\ldots,X_d)$ is Markov with respect to $G$.

\begin{definition}
\label{rep-dir}
{\em The random vector $({\bf p}(\i))_{\i\in\I}$ associated with the graph $G$ has a $\p$-Dirichlet distribution if the random vectors $({\bf p}^{v|\p_v}_{m|\k})_{m\in\I_v}$,  $\k\in\I_{\p_v}$, $v\in V$,  in  representation \eqref{rep} are independent and follow (classical) Dirichlet distributions.}
\end{definition}

Recall that if a random vector $({\bf p}^{v|\p_v}_{m|\k})_{m\in \I_v}$ has a classical Dirichlet distribution $\mathrm{Dir}(\alpha^{v|\p_v}_{m|\k},\,m\in\I_v)$,  it has the density
$$
f^{v|\p_v}_{\k}(\x)=\tfrac{\Gamma\left(\sum_{m=1}^{|\I_v|}\,\alpha^{v|\p_v}_{m|\k}\right)}
{\prod_{m=1}^{|\I_v|}\,\Gamma\left(\alpha^{v|\p_v}_{m|\k}\right)}\,\prod_{m=1}^{|\I_v|}\,x_m^{\alpha^{v|\p_v}_{m|\k}-1}\,
I_{T_{|\I_v|}}(\x),
$$

\vspace{2mm}\noindent
where $x_{|I_v|}=1-\sum_{m=1}^{|\I_v|-1}\,x_m$ and $T_{n+1}=\{(x_1,\ldots,x_n)\in(0,1)^n:\,\sum_{i=1}^n\,x_i<1\}$.

Moreover, since we assume that the random vectors $({\bf p}^{v|\p_v}_{m|\k},\,m\in \I_v)$, $\k\in\I_{\p_v}$, $v\in V$, are independent the joint density of $$({\bf p}^{v|\p_v}_{m|\k},\,m\in \I_v, \;\k\in\I_{\p_v},\;v\in V)$$ has the form

$$
f\left(\x^{v,\k},\,\k\in\I_{\p_v},\,v\in V\right)=
\prod_{v\in V}\,\prod_{\k\in\I_{\p_v}}\,\tfrac{\Gamma\left(\sum_{m\in\I_v}\,\alpha^{v|\p_v}_{m|\k}\right)}
{\prod_{m\in\I_v}\,\Gamma\left(\alpha^{v|\p_v}_{m|\k}\right)}\,
\prod_{m\in\I_v}\,\left(x_m^{v,\k}\right)^{\alpha^{v|\p_v}_{m|\k}-1},
$$

\vspace{2mm}\noindent
where the support is a cartesian product of unit simplexes  {\huge $\times$}$_{v\in V}\,\,T_{|\I_v|}^{\times\,|\I_{\p_v}|}$, that is $\x^{v,\k}\in T_{|\I_v|}$ and $x^{v,\k}_{|\I_v|}=1-\sum_{m=1}^{|\I_v|-1}\,x^{v,\k}_m$, $\k\in\I_{\p_v},\,v\in V$.

Now we define a $\cP$-Dirichlet distribution, where $\cP$ is a family of DAGs with skeleton $G$.

\begin{definition}
\label{def02}
{\em Let $\cP$ be a family of DAGs with skeleton $G$. The random vector $({\bf p}(\i),\,\i\in\I)$ associated with the graph $G$ has the $\cP$-Dirichlet distribution iff it has $\p$-Dirichlet distribution  for any $\p\in\cP$.}
\end{definition}

Of course, this definition implies that some consistency conditions for the parameters of the Dirichlet distributions defining the $\p$-Dirichlet distributions for the various $\p\in\cP$, have to be satisfied. This issue is conveniently treated by looking at the moments of the $\p$-Dirichlet and $\cP$-Dirichlet laws.

\begin{example}\label{exx}
Let $G=(V,E)$ be the decomposable graph with $V=\{1,2,3,4,5\}$ and cliques $\{1,2,5\}$, $\{2,3,5\}$ and $\{3,4,5\}$. Let $\cP=\{\p,\,\p'\}$ with
$$
\p_1=\{2,5\},\;\p_2=\emptyset, \; \p_3=\{2,5\},\;\p_4=\{3,5\},\;\p_5=\{2\}
$$
and
$$
\p'_1=\{2,5\},\;\p'_2=\{3,5\},\;\p'_3=\emptyset,\;\p'_4=\{3,5\},\;\p'_5=\{3\}.
$$
\bigskip
\begin{figure}[h]
\hskip -0.5cm

\psmatrix[mnode=circle,colsep=0.2,rowsep=1.5] 1 & & & & & 2 &  & & 3 & & & & & 4 & & & & & & 1 & & & & & 2 & & & 3 & & & & & 4 \\
 & & & & & & 5 & & & & & & & & & & & & & & & & & & & & 5 & & & & & &
\endpsmatrix
\psset{shortput=nab,arrows=-,labelsep=3pt}
\ncline{<-}{1,1}{1,6}
\ncline{->}{1,6}{1,9}
\ncline{->}{1,9}{1,14}
\ncline{<-}{1,1}{2,7}
\ncline{->}{1,6}{2,7}
\ncline{<-}{1,9}{2,7}
\ncline{<-}{1,14}{2,7}
\ncline{<-}{1,20}{1,25}
\ncline{<-}{1,25}{1,28}
\ncline{->}{1,28}{1,33}
\ncline{->}{2,27}{1,20}
\ncline{->}{2,27}{1,25}
\ncline{<-}{2,27}{1,28}
\ncline{->}{2,27}{1,33}
\vskip 0.3cm
\quad DAG $\mathfrak{p}$ in Example \ref{exx}.\quad\quad\quad\quad\quad\hspace{3mm}\quad\quad DAG $\mathfrak{p}'$ in Example \ref{exx}.
\end{figure}

\vspace{2mm}
Then the $\p$-Dirichlet distribution is defined up to a multiplicative constant as a product of the following independent Dirichlets:

$$
{\bf p}^2\sim \mathrm{Dir}(\alpha^2_m,\,m\in\I_2),\;\;{\bf p}^{5|2}_k\sim \mathrm{Dir}(\alpha^{5|2}_{m|k},\;m\in\I_5),\;k\in\I_2,\;{\bf p}^{1|25}_{\k}\sim\mathrm{Dir}(\alpha^{1|25}_{m|\k},\,m\in\I_1),\;\k\in\I_{25},
$$$${\bf p}^{3|25}_{\k}\sim \mathrm{Dir}(\alpha^{3|25}_{m|\k},\;m\in\I_3),\;\k\in\I_{25},\;{\bf p}^{4|35}_{\k}\sim \mathrm{Dir}(\alpha^{4|35}_{m|\k},\;m\in\I_4),\;\k\in\I_{35}
$$

\vspace{2mm}\noindent
and the $\p'$-Dirichlet is defined through

$$
{\bf p}^3\sim \mathrm{Dir}(\beta^3_m,\,m\in\I_3),\;\;{\bf p}^{5|3}_k\sim \mathrm{Dir}(\beta^{5|3}_{m|k},\;m\in\I_5),\;k\in\I_3,\;{\bf p}^{1|25}_{\k}\sim\mathrm{Dir}(\beta^{1|25}_{m|\k},\,m\in\I_1),\;\k\in\I_{25},
$$$${\bf p}^{2|35}_{\k}\sim \mathrm{Dir}(\beta^{2|35}_{m|\k},\;m\in\I_2),\;\k\in\I_{35},\;{\bf p}^{4|35}_{\k}\sim \mathrm{Dir}(\beta^{4|35}_{m|\k},\;m\in\I_4),\;\k\in\I_{35}.
$$
We will come back to this example later in subsection 4.2.1 and see how the constraints on the hyper parameters, i.e., the parameters of the $\cP$-Dirichlet come about.
\end{example}

\section{Moments}
\label{moments}
\subsection{The $\p$-Dirichlet distribution}
If the vector of random probabilities $({\bf p}(\i),\,\i\in\I)$ associated with the graph $G$ follows the $\p$-Dirichlet distribution, then for any non-negative integers $r_{\i}$, $\i\in\I$,

$$
\E\,\prod_{\i\in\I}\,\left[{\bf p}(\i)\right]^{r_{\i}}\,=\,\prod_{v\in V}\,\prod_{\k\in\I_{\p_v}}\,\E\,\prod_{m\in\I_v}\,\left[{\bf p}^{v|\p_v}_{m|\k}\right]^{r^{\q_v}_{\k,m}}
=\,\prod_{v\in V}\,\prod_{\k\in\I_{\p_v}}\,\tfrac{\Gamma(\sum_{m\in\I_v}\,\alpha^{v|\p_v}_{m|\k})}{\prod_{m\in\I_v}\,\Gamma(\alpha^{v|\p_v}_{m|\k})}\,
\tfrac{\prod_{m\in\I_v}\,\Gamma(\alpha^{v|\p_v}_{m|\k}+r^{\q_v}_{\k,m})}{\Gamma(\sum_{m\in\I_v}\,\alpha^{v|\p_v}_{m|\k}+r^{\p_v}_{\k})},
$$
where
$$
r^{\q_v}_{\k,m}=\sum_{\i\in\I:\,i_{\q_v}=(\k,m)}\,r_{\i}\quad \mbox{and}\quad r^{\p_v}_{\k}=\sum_{\i\in\I:\,i_{\p_v}=\k}\,r_{\i}.
$$

That is
\bel{p-moments}
\E\,\prod_{\i\in\I}\,\left[{\bf p}(\i)\right]^{r_{\i}}\,=\,\prod_{v\in V}\,\prod_{\k\in\I_{\p_v}}\,\tfrac{\prod_{m\in\I_v}\,\left(\alpha_{m|\k}^{v|\p_v}\right)^{r^{\q_v}_{\k,m}}}{\left(\tilde{\alpha}_{\k}^{\p_v}\right)^{r^{\p_v}_{\k}}},
\ee
where we write the rising factorial power  $\alpha(\alpha+1)\ldots (\alpha+r-1)$ as
$$(\alpha)^r=\frac{\Gamma(\alpha+r)}{\Gamma(\alpha)}$$
and
\bel{summ}
\tilde{\alpha}_{\k}^{\p_v}=\sum_{m\in\I_v}\,\alpha_{m|\k}^{v|\p_v}.
\ee

Note that since the $\p$-distribution has a bounded support it is uniquely determined by the moments as given in \eqref{p-moments}.

\subsection{The $\cP$-Dirichlet distribution}
Let $G=(V,E)$ be a decomposable graph as above. Denote by $\mathcal{C}$ the set of its cliques and by $\mathcal{S}$ the set of its separators.

If a vector of random probabilities $({\bf p}(\i),\,\i\in\I)$ associated with the graph $G$ follows the $\cP$-Dirichlet distribution then the formula for moments \eqref{p-moments} holds for all $\p\in\cP$.

Since its support is bounded,  the $\cP$-Dirichlet distribution  can be defined through the form of its moments.
The equality of the moments for the different representations of the $\p$-Dirichlet, $p\in\cP$, will impose equality constraints on the parameters of the $\p$-Dirichlet's and consequently those of the $\cP$-Dirichlet. Before developing the theory let us illustrate the mechanism using Example \ref{exx} above.

\subsubsection{Example \ref{exx} continued} Given nonnegative integers $r_{\i}$, $\i\in\I$, for any $D\subset V$ we define $$r^D_{\m}=\sum_{\i\in\I:\,\i_D=\m}\,r_{\i}.$$ For $D=\emptyset$ we write $r^{\emptyset}=r$. The equality of the moments obtained from \eqref{p-moments} for both $\p$ and $\p'$ yields \footnotesize
\begin{eqnarray*}
\E\,\prod_{\i\in\I}\,{\bf p}(\i)^{r_{\i}}&=&
\tfrac{\prod_{m\in \I_2}\,\left(\alpha^2_m\right)^{r^2_m}}{\left(\tilde{\alpha}\right)^r}\,\prod_{n\in\I_2}\,\tfrac{\prod_{m\in \I_5}\,\left(\alpha^{5|2}_{(m|n)}\right)^{r^{25}_{(n,m)}}}{\left(\tilde{\alpha}^2_n\right)^{r^2_n}}\,\prod_{\n\in\I_{25}}\,\tfrac{\prod_{m\in \I_1}\,\left(\alpha^{1|25}_{(m|\n)}\right)^{r^{125}_{(m,\n)}}}{\left(\tilde{\alpha}^{25,1}_{\n}\right)^{r^{25}_{\n}}}\\
&&\hspace{1.5cm}\times \prod_{\n\in\I_{25}}\,\tfrac{\prod_{m\in \I_3}\,\left(\alpha^{3|25}_{(m|\n)}\right)^{r^{325}_{(m,\n)}}}{\left(\tilde{\alpha}^{25,3}_{\n}\right)^{r^{25}_{\n}}}\,
\prod_{\n\in\I_{35}}\,\tfrac{\prod_{m\in \I_4}\,\left(\alpha^{4|35}_{(m|\n)}\right)^{r^{435}_{(m,\n)}}}{\left(\tilde{\alpha}^{35}_{\n}\right)^{r^{35}_{\n}}}\\
&=&
\tfrac{\prod_{m\in \I_3}\,\left(\beta^3_m\right)^{r^3_m}}{\left(\tilde{\beta}\right)^r}\,\prod_{n\in\I_3}\,\tfrac{\prod_{m\in \I_5}\,\left(\beta^{5|3}_{(m|n)}\right)^{r^{35}_{(n,m)}}}{\left(\tilde{\beta}^3_n\right)^{r^3_n}}\,\prod_{\n\in\I_{35}}\,\tfrac{\prod_{m\in \I_2}\,\left(\beta^{2|35}_{(m|\n)}\right)^{r^{235}_{(m,\n)}}}{\left(\tilde{\beta}^{35,2}_{\n}\right)^{r^{35}_{\n}}}\\
&&\hspace{1.5cm}\times\prod_{\n\in\I_{35}}\,\tfrac{\prod_{m\in \I_4}\,\left(\beta^{4|35}_{(m|\n)}\right)^{r^{435}_{(m,\n)}}}{\left(\tilde{\beta}^{35,4}_{\n}\right)^{r^{35}_{\n}}}\,
\prod_{\n\in\I_{25}}\,\tfrac{\prod_{m\in \I_1}\,\left(\beta^{1|25}_{(m|\n)}\right)^{r^{125}_{(m,\n)}}}{\left(\tilde{\beta}^{25}_{\n}\right)^{r^{25}_{\n}}}.
\end{eqnarray*}
\normalsize
Since there are no factorial powers in $r^2_m$ on the right-hand side of the equation above the terms in $r^2_m$ on the left-hand side must cancel out, that is $\alpha^2_m=\tilde{\alpha}^2_m$. Similarly, $\beta^3_m=\tilde{\beta}^3_m$. The factorial power $r^{125}_{\n}$ on the right- and left-hand side must be the same and therefore $\alpha^{1|25}_{(m,\n)}=\beta^{1|25}_{(m,\n)}$. Similarly, $\alpha^{3|25}_{(m,\n)}=\beta^{2|35}_{(m,\n)}$, $\alpha^{4|35}_{(m,\n)}=\beta^{4|35}_{(m,\n)}$ and also $\tilde{\alpha}=\tilde{\beta}$. For the factorial powers in $r^{25}_{\n}$ we observe that on the left-hand side there is one power in the numerator and two in the denominator, while on the right-hand side there is only one power in the denominator. Therefore the factorial power in the numerator must cancel with one of the two factorial powers of $\tilde{\alpha}^{25,3}_{\n}$ or of $\tilde{\alpha}^{25,1}_{\n}$ in the denominator. This means that
\begin{itemize}
\item either we have the cancelation $\alpha^{5|2}_{\n}=\tilde{\alpha}^{25,3}_{\n}$ and therefore
\newline $\tilde{\beta}^{25}_{\n}=\tilde{\alpha}^{25,1}_{\n},\;\;\forall\,\n\in\I_{25}$
\item or we have the cancelation $\alpha^{5|2}_{\n}=\tilde{\alpha}^{25,1}_{\n}$ and therefore
\newline $\tilde{\beta}^{25}_{\n}=\tilde{\alpha}^{25,3}_{\n}\;\;\forall\,\n\in\I_{25}.$
\end{itemize}

\vspace{2mm}\noindent
The first choice means that we associate the separator $\{2,5\}$ with the clique $\{1,2,5\}$ while in the second we associate $\{2,5\}$ with the clique $\{2,3,5\}$. This two choices correspond to two different $\p$-perfect orders of the cliques:

\begin{eqnarray}
o^{(1)}_{\p}&=&(C_1=\{2,3,5\},\;C_2=\{1,2,5\},\;C_3=\{3,4,5\})\\
o^{(2)}_{\p}&=&(C_1=\{1,2,5\},\;C_2=\{2,3,5\},\;C_3=\{3,4,5\}),
\end{eqnarray}

\vspace{2mm}\noindent
respectively. Of course, we could also exchange the cliques $C_2$ and $C_3$ in both orders. What is important is the pairings $(\{2,5\},\{1,2,5\})$ or $(\{2,5\},\{2,3,5\})$, respectively.

Similarly, for the factorial powers $r^{35}_{\n}$ on the right-hand side one can choose to cancel the factorial power of $r^{35}_{\n}$ in the numerator with factorial powers of either $\tilde{\beta}^{35,2}_{\n}$ or $\tilde{\beta}^{35,4}_{\n}$. Consequently,

$$
\mbox{either}\quad \beta^{5|3}_{\n}=\tilde{\beta}^{35,2}_{\n}\quad\mbox{and therefore}\quad \tilde{\alpha}^{35}_{\n}=\tilde{\beta}^{35,4}_{\n}\quad\forall\,\n\in\I_{35},
$$

$$
\mbox{or}\quad\beta^{5|3}_{\n}=\tilde{\beta}^{35,4}_{\n}\quad\mbox{and therefore}\quad\tilde{\alpha}^{35}_{\n}=\tilde{\alpha}^{35,2}_{\n}\quad \forall\n\in\I_{35},
$$

\vspace{2mm}\noindent
which corresponds to the two $\p'$-perfect orders:

\begin{eqnarray}
o^{(1)}_{\p'}&=&(C_1=\{2,3,5\},\;C_2=\{3,4,5\},\;C_3=\{1,2,5\})\\
o^{(2)}_{\p'}&=&(C_1=\{3,4,5\},\;C_2=\{2,3,5\},\;C_3=\{1,2,5\}),
\end{eqnarray}

\vspace{2mm}\noindent
respectively. Again here we could exchange $C_2$ and $C_3$ in both cases. What is important are the pairings $(\{3,5\}, \{3,4,5\})$ and $(\{3,5\},\{2,3,5\})$, respectively.

From any of these cancelation possibilities we obtain the same formula of moments

$$
\E\,\prod_{\i\in\I}\,{\bf p}(\i)^{r_{\i}}=\tfrac{\prod_{\n\in\I_{125}}\,\left(\nu^{125}_{\n}\right)^{r^{125}_{\n}}\,\prod_{\n\in\I_{235}}\,\left(\nu^{235}_{\n}\right)^{r^{235}_{\n}}\,\prod_{\n\in\I_{345}}\,\left(\nu^{345}_{\n}\right)^{r^{345}_{\n}}}
{(\mu)^r\prod_{\m\in\I_{25}}\left(\mu^{25}_{\m}\right)^{r^{25}_{\m}}\,\prod_{\m\in\I_{35}}\,\left(\mu^{35}_{\m}\right)^{r^{35}_{\m}}},
$$

\vspace{2mm}\noindent
but with different constraints for the parameters:

either (I)
$$\mu=\sum_{\m\in\I_{125}}\,\nu^{125}_{\m}=\sum_{\m\in\I_{345}}\,\nu^{345}_{\m}(=\sum_{\m\in\I_{235}}\,\nu^{235}_{\m}),$$ $$\mu^{25}_{\n}=\sum_{m\in\I_1}\,\nu^{125}_{(m,\n)}=\sum_{m\in\I_3}\,\nu^{235}_{(m,\n)},$$ $$\mu^{35}_{\n}=\sum_{m\in\I_4}\,\nu^{345}_{(m,\n)}=\sum_{m\in\I_2}\,\nu^{235}_{(m,\n)},$$
corresponding to the family of orders $O_{\cP}=(o^{(2)}_{\p},\,o^{(2)}_{\p'})$.

or (II)
$$\mu=\sum_{\m\in\I_{125}}\,\nu^{125}_{\m}=\sum_{\m\in\I_{235}}\,\nu^{235}_{\m},$$ $$\mu^{25}_{\n}=\sum_{m\in\I_1}\,\nu^{125}_{(m,\n)}=\sum_{m\in\I_3}\,\nu^{235}_{(m,\n)},$$
$$\mu^{35}_{\n}=\sum_{m\in\I_4}\,\nu^{345}_{(m,\n)},$$
corresponding to the family of orders $O_{\cP}=(o^{(2)}_{\p},\,o^{(1)}_{\p'})$.

or (III)
$$\mu=\sum_{\m\in\I_{235}}\,\nu^{235}_{\m}=\sum_{\m\in\I_{345}}\,\nu^{345}_{\m},$$
$$\mu^{25}_{\n}=\sum_{m\in\I_1}\,\nu^{125}_{(m,\n)},$$
$$\mu^{35}_{\n}=\sum_{m\in\I_4}\,\nu^{345}_{(m,\n)}=\sum_{m\in\I_2}\,\nu^{235}_{(m,\n)},$$
corresponding to the family of orders $O_{\cP}=(o^{(1)}_{\p},\,o^{(2)}_{\p'})$.

or (IV)
$$\mu=\sum_{\m\in\I_{235}}\,\nu^{235}_{\m},$$
$$\mu^{25}_{\n}=\sum_{m\in\I_1}\,\nu^{125}_{(m,\n)},$$
$$\mu^{35}_{\n}=\sum_{m\in\I_4}\,\nu^{345}_{(m,\n)},$$
corresponding to the family of orders $O_{\cP}=(o^{(1)}_{\p},\,o^{(1)}_{\p'})$.

We note that we obtained four different families of distributions, that is as many as the number of combinations of pairs $(S_l,C_l)$, where $S_l=\{2,5\}$ and $\{3,5\}$ here. Of course, choices will multiply with the number of separators with different possible pairings. In fact, more generally, choices may multiply with the number of elements of $\pp$ with different possible pairings in $\qq$ (see definitions \eqref{pandq} below)  and also with the size of $\cP$.

In this example,  we see that we have the  poset $(IV)\to(II,III)\to(I)$ of families of $\cP$-Dirichlet distributions, with family $(IV)$ being the maximal family.

In the remainder of this section we will show that to each collection of orders
\bel{op}
O_{\cP}=(o_{\p},\,\p\in\cP:\;o_{\p}\;\mbox{is}\;\p-\mbox{perfect},\;\p\in\cP)
\ee
corresponds a family of $\cP$-Dirichlet distributions. Though in our example it is easy to see that $(IV)$ is the maximal family of the poset we are unable to prove that for any given $\cP$, there exists such a unique maximal family.

\subsubsection{The moment formula}

To find a convenient expression for the moment formula we need the following two auxiliary results.

\begin{lemma}\label{CS}
Consider a DAG with skeleton $G$ defined by a parent function $\p$. Then,
$$
\mathcal{C}\subset \q(V)\qquad \mbox{and}\qquad \mathcal{S}\subset\p(V).
$$
Moreover,
$$
\p(V)\setminus \mathcal{S}=\q(V)\setminus\mathcal{C}.
$$
\end{lemma}
\begin{proof}
Let $o=(C_1,\ldots,C_K)$ be a $\p$-perfect order of cliques. Using the notation $v_{l,j}$, $j=s_l+1,\ldots, c_l$, $l=1,\ldots,K$, introduced in \eqref{numb}, we see that $\p_{v_{l,s_l+1}}=S_l$ and $\q_{v_{l,c_l}}=C_l$, $l=1,\ldots,K$ (here we set $S_1=\emptyset$) so that the first statement of the lemma is proved.

For the second part note that for any $l=1,\ldots,K$ and any $j=s_l+2,\ldots,c_l$, we have $\q_{v_{l,j-1}}=\p_{v_{l,j}}\not \in\mathcal{C}\cup\mathcal{S}$. That is
$$
\p(V)\setminus \mathcal{S}=\q(V)\setminus\mathcal{C}=\bigcup_{l=1}^K\,\bigcup_{j=s_l+2}^{c_l}\,\p_{v_{l,j}}.
$$
\end{proof}

Note that if $v,w\in V$ are distinct then $\q_v$ and $\q_w$ are also distinct, but $\p_v$ and $\p_w$ may be the same. It means that while considering the set $\p(V)$ we allow for a given set to appear in $\p(V)$ more than once, that is, each element of $\p(V)$ has its multiplicity, e.g. a set $S$ which is a separator can happen to be a multiple separator, moreover such a set $S$ can appear, at most once (due to Lemma \ref{CS}), in $\p(V)\setminus \ss$.

Denote
$$
\mathfrak{R}_{\p}:=\p(V)\setminus \mathcal{S}=\q(V)\setminus\mathcal{C}.
$$

Let $\cP$ be a family of DAGs with the same skeleton $G$. Since in a DAG in $\cP$ there are no immoralities, they are all Markov equivalent.
Consider the following sets
\bel{pandq}
\mathfrak{Q}=\bigcap_{\p\in\cP}\,\q(V)\supset\mathcal{C},\qquad\mbox{and} \qquad \mathfrak{P}=\bigcap_{\p\in\cP}\,\p(V)\supset \mathcal{S},
\ee
where in the definition of $\mathfrak{P}$ we allow for multiple separators. By Lemma \ref{CS} it follows that
$$
\mathfrak{Q}=\mathcal{C}\cup\mathfrak{R}\qquad\mbox{and}\qquad \mathfrak{P}=\mathcal{S}\cup\mathfrak{R},
$$
where $\mathfrak{R}=\bigcap_{\p\in\cP}\,\mathfrak{R}_{\p}$.


\begin{lemma}
\label{lemmaq}
For a family $\cP$ of DAGs with  skeleton $G$, let $O_{\cP}$ be a collection of $\p$-perfect orders of cliques, $\p\in\cP$.  For any clique $C$, let $\mathfrak{R}_C$ denote a family of these elements of $\mathfrak{R}$ which are contained in $C$. If there exists $o\in O_{\cP}$ such that $\ss\ni S\stackrel{o}{\to}C$ then all the elements $Q^C_1,\ldots,Q^C_{j_C-1}\in\mathfrak{R}_C$ (it may be empty) can be numbered as follows

\begin{equation}\label{seqq}
S=:Q^C_{j_C}\varsubsetneq Q^C_{j_C-1}\varsubsetneq  Q^C_{j_C-2}\varsubsetneq  \ldots \varsubsetneq  Q^C_2\varsubsetneq  Q^C_1\varsubsetneq Q^C_0:=C.
\end{equation}
\end{lemma}

\vspace{2mm}\noindent
\begin{proof}
Let $o$ be $\p$-perfect for a $\p\in\cP$. The result follows immediately from the fact (see \eqref{numb} and the proof of Lemma \ref{CS}) that possible sets $Q^C_i$, $1\le i\le j_C-1$, from $\mathfrak{R}_C$  are of the form $\q_{v_{l,k}}=\p_{v_{l,k+1}}$, $k=s_l+1,\ldots,c_l-1$, moreover $S=\p_{v_{l,s_l+1}}\varsubsetneq \q_{v_{l,s_l+1}}$ and $C=\q_{v_{l,c_l}}\varsupsetneq \q_{v_{l,c_l-1}}$, $l=1,\ldots,K$,
\end{proof}

Now we are in a position to give the formula for moments.
\begin{theorem}\label{momenty}
A vector of random probabilities $({\bf p}(\i),\,i\in\I)$ associated with the graph $G$ has a $\cP$-Dirichlet distribution iff there exists a collection $O_{\cP}$ of $\p$-perfect orders, $\p\in\cP$, as in \eqref{op}, such that for any $\,r_{\i}\in\N, \;\i\in\I$,

\bel{epr}
\E\,\prod_{\i\in\I}({\bf p}(\i))^{r_{\i}}=\tfrac{\prod_{A\in\qq}\,\prod_{\m\in\I_A}\,(\nu^A_{\m})^{r^A_{\m}}}
{\prod_{B\in\pp}\,\prod_{\n\in\I_B}\,(\mu^B_{\n})^{r^B_{\n}}}\qquad
\ee

\vspace{2mm}\noindent
where $\nu^A_{\m}$, $\m\in\I_A$, $A\in\qq$, and $\mu^B_{\n}$, $\n\in\I_B$, $B\in\pp$, are positive numbers satisfying

\bel{constr}
\mu_{\n}^B=\sum_{\k\in\I_{A\setminus B}}\,\nu^A_{(\n,\k)}\qquad\forall\,\n\in\I_B
\ee

\vspace{2mm}\noindent
whenever there exist $\ss\ni S\subset C\in\cc$ and $o\in O_{\cP}$ such that $S\stackrel{o}{\to}C$ and $B=Q^C_i\subsetneq A=Q^C_{i-1}$ for some $i\in\{1,\ldots,j_C\}$.
\end{theorem}

The proof is given in the Appendix.

It follows from the expression  \eqref{epr}  of the moments that the $\cP$-Dirichlet distribution has a density which, following a given perfect order $o\in O_{\cP}$ of the cliques, can be expressed as the product of independent classical Dirichlet distributions as given below. Using the notation $R^{C_l}_i=Q^{C_l}_i\setminus Q^{C_l}_{i+1}, \;i\in\{0,\ldots,j_{C_l}-1\}$, $l\in\{1,\ldots,K\}$, we have the following.
\begin{corollary}
\label{denscp}
Let  $({\bf p}(\i),\,\i\in\I)$ be a random vector having $\cP$-Dirichlet distribution with constraints \eqref{constr} governed by a family $O_{\cP}$. Consider any perfect order  $o=(C_1,\ldots,C_K)\in O_{\cP}$ of the cliques. There exist independent vectors $({\bf p}^{R^{C_l}_i|Q^{C_l}_{i+1}}_{\m|\k},\,\m\in \I_{R^{C_l}_i})$ having classical Dirichlet distributions $\mathrm{Dir}(\nu_{(\k,\m)}^{Q_i^{C_l}},\m\in\I_{R^{C_l}_i})$, $\k\in \I_{Q^{C_l}_{i+1}}$, $i\in\{0,1,\ldots,j_{C_l}-1\}$, $l\in\{1,\ldots,K\}$, such that for any $\i\in\I$

$$
{\bf p}(\i)=\prod_{l=1}^K\,\prod_{i=0}^{j_{C_l}-1}\,{\bf p}^{R^{C_l}_i|Q^{C_l}_{i+1}}_{\m|\k},\quad \mbox{where}\;\;\m=\i_{R^{C_l}_i}\;\;\mbox{and}\;\;\k=\i_{Q^{C_l}_{i+1}}.
$$

\vspace{2mm}\noindent
Thus the density of the $\cP$-Dirichlet distribution can be written as

\begin{equation}
\label{analoghp}
\prod_{l=1}^K  \,\prod_{i=0}^{j_{C_l}-1}\, \prod_{\k\in \I_{Q^{C_l}_{i+1}}}\,\mathrm{Dir}(\nu^{Q^{C_l}_i}_{(\k,\m)},\m\in\I_{R^{C_l}_i}).
\end{equation}
\end{corollary}

\vspace{2mm}

\noindent We note that the well-known decomposition of the hyper Dirichlet density \eqref{hp} into
  \begin{eqnarray*}
C \prod_{l=1}^K\,\prod_{\k\in \I_{S_l}}\,\prod_{\m\in \I_{R_l}}            ({\bf p}^{R_l|S_l}_{\m|\k})^{\alpha^{C_l}_{(\m,\k)}-1},
\end{eqnarray*}
where $C$ is the normalizing constant, is a special case of \eqref{analoghp} where $j_c=1$ and $Q^{C_l}_{i+1}$ and $Q^{C_l}_i$ are respectively replaced by $S_l$ and $R_l=C_l\setminus S_l$.

\subsubsection{More examples and the hyper-Dirichlet as a special case of the ${\cP}$-Dirichlet} We now give a few examples of the $\cP$-Dirichlet distribution. We start with an example in which $\mathfrak{R}\neq\emptyset$.
\begin{example}
\label{41}
Let $G=(V,E)$ be a graph with $$V=\{1,2,3,4,5\}\quad\mbox{and}\quad E=\{\{1,3\},\,\{2,4\},\,\{3,4\},\,\{3,5\},\,\{4,5\}\}.$$ Then $\cc=\{\{1,3\},\,\{3,4,5\},\,\{2,4\}\}$ and $\ss=\{\emptyset,\,\{3\},\,\{4\}\}$.

Let $\mathfrak{P}=\{\p,\,\p'\}$, where
$$
\p_1=\emptyset,\;\;\p_2=\{4\},\;\;\p_3=\{1\},\;\;\p_4=\{3\},\,\,\p_5=\{3,4\}
$$
and
$$
\p'_1=\{3\},\;\;\p'_{2}=\emptyset,\;\;\p'_{3}=\{4\},\;\;\p'_4=\{2\},\;\;\p'_5=\{3,4\}.
$$

\bigskip
\begin{figure}[h]
\hskip -0.5cm
\psmatrix[mnode=circle,colsep=0.2,rowsep=1.5] 1 & & & & & 3 &  & & 4 & & & & & 2 & & & & & & 1 & & & & & 3 & & & 4 & & & & & 2 \\
 & & & & & & 5 & & & & & & & & & & & & & & & & & & & & 5 & & & & & &
\endpsmatrix
\psset{shortput=nab,arrows=-,labelsep=3pt}
\ncline{->}{1,1}{1,6}
\ncline{->}{1,6}{1,9}
\ncline{->}{1,9}{1,14}
\ncline{->}{1,6}{2,7}
\ncline{->}{1,9}{2,7}
\ncline{<-}{1,20}{1,25}
\ncline{<-}{1,25}{1,28}
\ncline{<-}{1,28}{1,33}
\ncline{<-}{2,27}{1,25}
\ncline{<-}{2,27}{1,28}
\vskip 0.3cm
DAG $\mathfrak{p}$ for Example \ref{41}.\quad\quad\quad\quad\quad\quad\hspace{3mm}\quad DAG $\mathfrak{p}'$ for Example \ref{41}.
\end{figure}

\vspace{2mm}
Then $$\mathfrak{Q}=\mathcal{C}\cup\{3,4\},\qquad \mathfrak{P}=\mathcal{S}\cup\{3,4\},$$
$$\mathfrak{R}_{\{3,4,5\}}=\{3,4\},\quad \mathfrak{R}_{\{1,3\}}=\mathfrak{R}_{\{2,4\}}=\emptyset.$$

Moreover, there is only one available collection of orders $O_{\cP}=\{o,o'\}$, where the $\p$-perfect order $o$ is the following $C_1=\{1,3\}$, $C_2=\{3,4,5\}$, $C_3=\{2,4\}$ and the $\p'$-perfect order $o'$ is the following $C_1'=\{2,4\}$, $C_2'=\{3,4,5\}$, $C_3'=\{1,3\}$.

Then formula \eqref{epr}  for moments becomes

$$
\E\,\prod_{\i\in\I}\,\left[{\bf p}(\i)\right]^{r_{\i}}=\tfrac{\prod_{\m\in\I_{13}}\,\left(\nu^{13}_{\m}\right)^{r^{13}_{\m}}\,
\prod_{\m\in\I_{24}}\,\left(\nu^{24}_{\m}\right)^{r^{24}_{\m}}\,\prod_{\m\in\I_{345}}\,\left(\nu^{345}_{\m}\right)^{r^{345}_{\m}}\,\prod_{\m\in\I_{34}}
\,\left(\nu^{34}_{\m}\right)^{r^{34}_{\m}}}
{(\mu)^r\,\prod_{m\in\I_3}\,\left(\mu^3_m\right)^{r^3_m}\,\prod_{m\in\I_4}\,\left(\mu^4_m\right)^{r^4_m}
\,\prod_{\m\in\I_{34}}\,\left(\mu^{34}_{\m}\right)^{r^{34}_{\m}}}$$

\vspace{2mm}\noindent
with the following consistency conditions:
$$
\mu=\sum_{\k\in\I_{13}}\,\nu^{13}_{\k}=\sum_{\k\in\I_{24}}\,\nu^{24}_{\k};
$$
$$
\mu^3_n=\sum_{k\in\I_4}\,\nu^{34}_{(n,k)}=\sum_{k\in\I_1}\,\nu^{13}_{(k,n)},\qquad n\in\I_3;
$$
$$
\mu^4_n=\sum_{k\in\I_3}\,\nu^{34}_{(k,n)}=\sum_{k\in\I_2}\,\nu^{24}_{(k,n)},\qquad n\in\I_4;
$$
$$
\mu^{34}_{\n}=\sum_{k\in\I_5}\,\nu^{345}_{(\n,k)},\qquad \n\in\I_{34}.
$$
Consequently, combining the above equations, we also get
$$
\mu=\sum_{n\in\I_3}\,\mu^3_n=\sum_{n\in\I_4}\,\mu^4_n=\sum_{\k\in\I_{34}}\,\nu^{34}_{\k}.
$$
\end{example}

We will now look at the case where $\mathfrak{R}=\emptyset$. We will see that the formula for moments simplifies and becomes closer to the moment formula for the hyper Dirichlet. Let us first  consider an example.

\begin{example}\label{42}
Consider a tree $G=(V,E)$ with $V=\{1,2,3,4,5\}$ and $E=\{\{1,2\},\,\{2,3\},\,\{2,4\},\,\{4,5\}\}$. Let $\cP=\{\p,\,\p'\}$, where
$$
\p_1=\emptyset,\;\;\p_2=\{1\},\;\;\p_3=\{2\},\;\;\p_4=\{2\},\,\,\p_5=\{4\}
$$
and
$$
\p'_1=\{2\},\;\;\p'_{2}=\{3\},\;\;\p'_{3}=\emptyset,\;\;\p'_4=\{2\},\;\;\p'_5=\{4\}.
$$

\bigskip
\begin{figure}[h]
\hskip -0.5cm

\psmatrix[mnode=circle,colsep=0.2,rowsep=1.5] 1 & & & & & 2 &  & & 4 & & & & & 5 & & & & & & 1 & & & & & 2 & & & 4 & & & & & 5 \\
 & & & & & & 3 & & & & & & & & & & & & & & & & & & & & 3 & & & & & &
\endpsmatrix
\psset{shortput=nab,arrows=-,labelsep=3pt}
\ncline{->}{1,1}{1,6}
\ncline{->}{1,6}{1,9}
\ncline{->}{1,9}{1,14}
\ncline{->}{1,6}{2,7}
\ncline{<-}{1,20}{1,25}
\ncline{->}{1,25}{1,28}
\ncline{->}{1,28}{1,33}
\ncline{->}{2,27}{1,25}
\vskip 0.3cm \quad\quad DAG $\mathfrak{p}$ for Example \ref{42}.\quad\quad\hspace{3mm}\quad\quad\quad\quad\quad DAG $\mathfrak{p}'$ for Example \ref{42}.
\end{figure}

\vspace{2mm}
Then $$\mathfrak{Q}=\mathcal{C}=E,\qquad \mathfrak{P}=\{\{2\},\{2\},\,\{4\}\}=\mathcal{S}.$$

Moreover $\mathfrak{R}_C=\emptyset$ for any $C\in\mathcal{C}$. For $S=\{4\}$ and $C=\{2,4\}$ it never happens that $S=S_l$ and $C=C_l$ for perfect orders of cliques $C_1,C_2,C_3,C_4$ generated by $\p$ or $\p'$.

That is the formula for moments as in \eqref{epr} reads

$$
\E\,\prod_{\i\in\I}\,\left[{\bf p}(\i)\right]^{r_{\i}}=\tfrac{\prod_{\m\in\I_{12}}\,\left(\nu^{12}_{\m}\right)^{r^{12}_{\m}}\,
\prod_{\m\in\I_{23}}\,\left(\nu^{23}_{\m}\right)^{r^{23}_{\m}}\,\prod_{\m\in\I_{24}}\,\left(\nu^{24}_{\m}\right)^{r^{24}_{\m}}\,\prod_{\m\in\I_{45}}\,\left(\nu^{45}_{\m}\right)^{r^{45}_{\m}}}
{(\mu)^r\,\prod_{m\in\I_2}\,\left(\mu^2_m\right)^{2r^2_m}\,\prod_{m\in\I_4}\,\left(\mu^4_m\right)^{r^4_m}},
$$

\vspace{2mm}\noindent
with the following consistency conditions: either we consider the set of orders: $\p$-perfect $o$ of the form $C_1=\{1,2\}$, $C_2=\{2,3\}$, $C_3=\{2,4\}$, $C_4=\{4,5\}$ and $\p'$-perfect $o'$ of the form $C_1'=\{2,3\}$, $C_2'=\{1,2\}$, $C_3'=\{2,4\}$, $C_4'=\{4,5\}$ and then the constraints
$$
\mu=\sum_{\k\in\I_{12}}\,\nu^{12}_{\k}=\sum_{\k\in\I_{23}}\,\nu^{23}_{\k},
$$
$$
\mu^{2(1)}_n=\sum_{k\in\I_1}\,\nu^{12}_{(k,n)}=\sum_{k\in\I_3}\,\nu^{23}_{(n,k)},\quad n\in\I_2,
$$
$$
\mu^{2(2)}_n=\sum_{k\in\I_4}\,\nu^{24}_{(n,k)},\quad n\in\I_2,
$$
$$
\mu^{4}_n=\sum_{k\in\I_5}\,\nu^{45}_{(n,k)},\quad n\in\I_4.
$$
or we keep $o$ as defined above and for $o'$ we take $C_1'=\{2,3\}$,  $C_2'=\{2,4\}$, $C_3'=\{4,5\}$ $C_4'=\{1,2\}$ and the constraints are more restrictive since while the first and the fourth line above are preserved the second and the third merge into
$$
\mu^{2(1)}_n=\mu^{2(2)}_n=\sum_{k\in\I_1}\,\nu^{12}_{(k,n)}=\sum_{k\in\I_3}\,\nu^{23}_{(n,k)}=\sum_{k\in\I_4}\,\nu^{24}_{(n,k)},\quad n\in\I_2.
$$

\end{example}

The above example falls under a more general setting which follows immediately from Theorem \ref{momenty} and which we formalize as:

\begin{corollary}\label{cordir}Let a random vector $({\bf p}(\i))_{\i\in\I}$ associated with the graph $G$ have a $\cP$-Dirichlet distribution for a family $\cP$ of moral DAGs and a collection $O_{\cP}$ of $\p$-perfect orders, $\p\in\cP$. If $\cP$ has the property $\mathfrak{Q}=\mathcal{C}$ (that is, $\mathfrak{P}=\mathcal{S}$ or $\mathfrak{R}=\emptyset$) then
\bel{mhd}
\E\,\prod_{\i\in\I}\,\left[{\bf p}(\i)\right]^{r_{\i}}=\tfrac{\prod_{C\in \mathcal{C}}\,\prod_{\m\in\I_C}\,\left(\nu^C_{\m}\right)^{r^C_{\m}}}{(\mu)^r\,\prod_{S\in \mathcal{S}}\,\prod_{\m\in\I_S}\,\left(\mu^S_{\m}\right)^{r^S_{\m}}},
\ee
with the following consistency conditions
\bel{cons}
\sum_{\m\in C:\m_S=\n}\,\nu^C_{\m}=\mu^S_{\n},\qquad \n\in\I_S
\ee
whenever $S\stackrel{o}{\to}C$ for some $o\in O_{\cP}$.
\end{corollary}

The hyper-Dirichlet distribution  is, of course,  uniquely defined through moments of the form exactly the same as in \eqref{mhd} but the consistency conditions are even stronger: \eqref{cons} is satisfied for any $\mathcal{S}\ni S \subset C\in\mathcal{C}$, where $S=\emptyset\in\mathcal{S}$. Therefore a  $\cP$-Dirichlet distribution for the family $\cP$  of {\em all} moral DAGs with skeleton $G=(V,E)$ is a hyper-Dirichlet distribution. Actually,  it follows directly from Corollary \ref{cordir} above that a considerably smaller family $\cP$ forces the $\cP$-Dirichlet to be hyper-Dirichlet. We state this as follows.

\begin{theorem}\label{hyp-dir}
Let $\mathcal{L}$ be a $\cP$-Dirichlet distribution for a family $\cP$ of moral DAGs and a collection $O_{\cP}$ of $\p$-perfect orders, $\p\in\cP$. If
\bel{QQQ}
\mathfrak{R}=\emptyset
\ee
and  \bel{sos}\forall\,(S\in \ss,\,C\in\cc)\;\; \mbox{if}\;\; S\subset C\;\;\mbox{then}\;\;\exists\,o\in O_{\cP}\;\;\mbox{such that}\;\; S\stackrel{o}{\to}C\ee
then $\mathcal{L}$ is a hyper-Dirichlet distribution.
\end{theorem}
\section{The $\cP$-Dirichlet as a prior distribution}
\label{prior}
In this section, we  look at the properties of the  $\cP$-Dirichlet as a prior distribution.
We first compute  the dimension of the $\cP$-Dirichlet family for a given $\cP$ and show that it is always larger than the dimension of the  hyper Dirichlet family with the same skeleton $G$ (unless the set $\cP$ is so large that the corresponding $\mathfrak R$ is empty and \eqref{sos} holds, of course). We then show that the $\cP$-Dirichlet is a conjugate distribution with the directed strong hyper Markov property for every $\p\in \cP$. This, of course, makes the $\cP$-Dirichlet easy to use in a Bayesian model selection process. In Section 5.3, we argue that, when we have additional constraints on the direction of certain edges between variables,
 the $\cP$-Dirichlet arises naturally as  a conjugate prior when doing model selection in the class of models Markov with respect to a DAG with a given decomposable skeleton $G$ .
\subsection{Dimension of the $\cP$-Dirichlet family}
We are now going to show that the dimension of the parameter space of the $\cP$-Dirichlet distribution  is always greater (or equal)  than that of the hyper Dirichlet. This means, of course, that when choosing the $\cP$-Dirichlet as a prior rather than the hyper Dirichlet, we gain  flexibility in our choice of the hyper parameters. The dimensions of both families are given in the following theorem.

Let $G$ be a decomposable graph with ${\mathcal C}$ as its set of cliques and ${\mathcal S}$ as its set of separators.
If $\cP$ is a collection of DAG's with skeleton $G$ and $O_{\cP}$ a collection of $\p$-perfect orders, $\p\in\cP$, for $S\in {\mathcal S}$ given, we denote by $N_S$ be the number of cliques $C$ such that
 $$\mbox{if}\;\;S\subset C\;\;\mbox{then}\;\;\exists\,o\in O_{\cP}\;\;\mbox{such that}\;\; S\stackrel{o}{\to}C.$$
Recall for that for $v\in V$, $\I_v$ denotes the set of values that $X_v$ can take and $|\I_v|$ denotes the cardinality of $\I_v$.
\begin{theorem}
\label{count}
For $G$ and $\cP$ as given above, the dimension of the parameter space of the $\cP$-Dirichlet family of distributions is
\begin{equation}
\label{ncp}
{\mathcal N}_{\cP}=\sum_{Q\in {\mathfrak Q}}\prod_{v\in Q}|\I_v|-\sum_{S\in {\mathcal S}}(N_S-1)\prod_{v\in S}|\I_v|\;.
\end{equation}
The dimension of the parameter space of the hyper Dirichlet family of distributions with the same skeleton $G$ is equal to
\begin{equation}
\label{nhp}
{\mathcal N}_{HP}=\sum_{C\in {\mathcal C}}\prod_{v\in C}|\I_v|-\sum_{S\in {\mathcal S}}(N_S-1)\prod_{v\in S}|\I_v|\;.
\end{equation}
Moreover, if the $\cP$-Dirichlet is not identical to the hyper Dirichlet, we always have
\begin{equation}
\label{neq}
{\mathcal N}_{\cP}>{\mathcal N}_{HP}\;.
\end{equation}
\end{theorem}
\begin{proof}
From \eqref{epr} and \eqref{constr}, we see that the parameters are the $\nu^A_{\underline{m}}$ and we need not count the $\mu^B_{\underline{n}}$ since they are defined by the constraints of the type \eqref{constr}. There clearly are $\sum_{Q\in {\mathfrak Q}}\prod_{v\in Q}|\I_v|$ such parameters $\nu^A_{\underline{m}}$. They are not all free since an element $S\in {\mathcal S}$ can be equal to an element $Q^C_{j_C}$ for several $C\in {\mathcal C}$.
More precisely for all $C\in {\mathcal C}$ such that there exists $o\in O_{\cP}$ with $S\stackrel{o}{\to}C$, we would have
$$Q^C_{j_C}\varsubsetneq Q^C_{j_C-1}$$
and therefore by \eqref{constr}, we have $(N_S-1)$ equality of the type
$$\mu^S_{\n}=
\mu^{Q^C_{j_C}}_{\n}
=\sum_{\k\in I_{Q^C_{j_{C-1}}\setminus S }}\nu^{Q^C_{j_{C-1}}}_{(\n,\k)}\;$$
and thus $(N_S-1)$ constraints for a given $\nu^{Q^C_{j_{C-1}}}_{(\n,\k)}$ and thus a total of $(N_S-1)\prod_{v\in S}|\I_v|$ constraints for each $S$.
 We now note  that if $B\in {\mathfrak R}$,  that is if $B$ is not a separator, the corresponding equation \eqref{constr} is not a constraint since then there is only one clique to which $B$ can belong , i.e. only one $A$ such that
$$B=Q^C_i\varsubsetneq A=Q^C_{i-1}\;\;\mbox{and}\;\;\mu^B_{\n}=\sum_{\k\in I_{A\setminus B}}\nu^A_{(\n,\k)}\;.$$
It follows that \eqref{ncp} is proved.

 In the case of the hyper Dirichlet, a similar argument shows us that the total number of parameters is
 $\sum_{C\in {\mathcal C}}\prod_{v\in Q}|\I_v|$. The constraints given by equations of the type \eqref{constr} are of the form
$$\mu^S_{\underline{n}}=\sum_{\underline{k}\in C\setminus S}\nu^{C}_{(\n,\k)}$$
for any C containing $S$. Since considering the hyper Dirichlet is equivalent to taking $\cP$ as the set of all DAG's Markov equivalent to $G$, $N_S$ is nothing but the number of cliques containing $S$ and equation \eqref{nhp} follows.

To see that ${\mathcal N}_{\cP}$ is always strictly greater than ${\mathcal N}_{HP}$, we observe that unless $\mathfrak R=\emptyset$, that is the $\cP$-Dirichlet is the hyper Dirichlet, ${\mathcal C} \varsubsetneq  {\mathfrak Q}$ and therefore
$$\sum_{Q\in {\mathfrak Q}}\prod_{v\in Q}|\I_v|>\sum_{C\in {\mathcal C}}\prod_{v\in C}|\I_v|\;$$
Moreover, for each $S$, $N_S$ in the $\cP$-Dirichlet is less than or equal  to the corresponding $N_S$ in the hyper Dirichlet. Inequality \eqref{neq}  follows immediately.
 \end{proof}
 Let us illustrate this result by deriving the dimension of  $\cP$-Dirichlet and the hyper Dirichlet families respectively when $G$ and $\cP$ are as given in Example 4.1.
\begin{example}(Example 4.1 continued)
\label{41cont}
Let us assume that all variables are binary, that is $|\I_v|=2,\;v=1,2,3,4,5$. From the expression of the moments in \eqref{epr}, we see that the number of $\nu^A_{\m}$ parameters is equal to
4 for $(\nu^{13}_{\m}, \m\in \I_{13})$ plus 4 for $(\nu^{34}_{\m}, \m\in \I_{34})$ plus 4 for $(\nu^{24}_{\m}, \m\in \I_{24})$ plus 8 for  for $(\nu^{345}_{\m}, \m\in \I_{345})$. We therefore have a total of 20 parameters. The separators are  $\{3\}$ and $\{4\}$. We have
$$\{3\}\stackrel{o'}{\to} \{1,3\}\;\;\mbox{and} \;\;\{3\}\stackrel{o}{\to} \{3,4,5 \}$$
and therefore $N_{\{3\}}=2$. Similarly
$$\{4\}\stackrel{o}{\to} \{2,4\}\;\;\mbox{and} \;\;\{4\}\stackrel{o'}{\to} \{3,4,5 \}$$
and therefore $N_{\{4\}}=2$. According to \eqref{ncp}, the dimension of the $\cP$-Dirichlet family is
$${\mathcal N}_{\cP}=20-2-2=16.$$
For the hyper Dirichlet, the cliques $\{1,3\},\; \{3,4,5\}$ and $\{2,4\}$ yield respectively 4, 8 and 4 parameters for a total of 16 while the $N_S$ are the same as in the case of the $\cP$-Dirichlet and therefore, according to \eqref{nhp},
$${\mathcal N}_{HP}=16-2-2=12.$$
\end{example}
\subsection{Conjugacy and directed strong hyper Markov property}
We will now emphasize the properties of the $\cP$-Dirichlet that make it a useful prior for model selection in a restricted class of DAG's. In the following proposition, we state that for any $\p\in \cP$, the $\cP$-Dirichlet is strong directed hyper Markov and conjugate. We now recall the definition of the strong directed hyper Markov property.
Let $X=(X_1,\ldots,X_d)$ be a random variable Markov with respect to a DAG given by a parent function $\p$, with distribution parameterized by $\theta\in \R^k$ for some $k$, which itself follows a law ${\mathcal L}$.
Let $\theta_{\p_v}, \theta_{v|\p_v}$ and $\theta_{{nd}_v}$ denote respectively the parameters of the marginal distribution of $X_{\p_v}$, the conditional distribution of $X_v$ given $X_{\p_v}$ and the marginal distribution of the non descendants of $v$. Then the law ${\mathcal L}$ is said to be strong directed hyper Markov if we have the conditional independences
$$\theta_{v|\p_v}\perp \theta_{{nd}_v}|\;\theta_{\p_v}, \;v\in V.$$
With this definition, we see that the strong directed hyper Markov property of the $\cP$-Dirichlet follows by construction. We now state and prove that it is a conjugate family.

\begin{theorem}
Let the conditional distribution of cell counts ${\bf N}=({\bf N}(\i), \i\in \I)$ for ${\bf X}=(X_v, v\in V)$ given ${\bf p}=({\bf p}(\i), \i\in \I)$ be multinomial ${\mathcal M}_G({\bf p}(\i), \i\in \I)$ Markov with respect to the decomposable graph $G$. Let ${\bf p}$ follow a $\cP$-Dirichlet distribution with hyper parameters
$$\nu^{A}_{\m},\;\m\in \I_{A},\; A\in {\mathfrak Q}\qquad\mbox{and}\qquad
\mu^{B}_{\n},\; \n\in \I_{B},\; B\in {\mathfrak P}$$
as given in \eqref{epr} and \eqref{constr}. Then  the posterior distribution of ${\bf p}$ given ${\bf N}=(n(\i),\i\in\I)$  is $\cP$-Dirichlet with hyper parameters
$$n^A_{\m}+\nu^{A}_{\m},\;\m\in \I_A,\; A\in {\mathfrak Q}\qquad\mbox{and}\qquad
n^B_{\n}+\mu^{B}_{\n},\; \n\in \I_{B},\; B\in {\mathfrak P},$$
where $n^A_{\m}$ is the $A$-marginal count for $\i_A=\m$.

Moreover, for any $\p\in \cP$, the $\cP$-Dirichlet is strong hyper Markov.
\end{theorem}
\begin{proof}
The conditional distribution of ${\bf N}$ given ${\bf p}$ has the density (with respect to the counting measure) which, up to a multiplicative scalar, is equal to $\prod_{\i\in\I}\,[{\bf p}(\i)]^{N(\i)}$. Then, by the generalized Bayes formula for any table ${\bf r}=(r(\i),\,\i\in\I)$ of nonnegative integers
$$
\E\left(\left.\prod_{\i\in \I}\,[{\bf p}(\i)]^{r_{\i}}\right|{\bf N}=(n(\i),\,\i\in\I)\right)=\tfrac{\E\,\prod_{\i\in \I}\,[{\bf p}(\i)]^{r_{\i}+n(\i)}}{\E\,\prod_{\i\in \I}\,[{\bf p}(\i)]^{n(\i)}}.
$$
Applying \eqref{epr} to the numerator and denominator we see that the right hand side above can be written as
$$
\tfrac{\prod_{A\in\qq}\,\prod_{\m\in\I_A}\,(\nu^A_{\m})^{r^A_{\m}+n^A_{\m}}}{\prod_{A\in\qq}\,\prod_{\m\in\I_A}\,(\nu^A_{\m})^{n^A_{\m}}}
\tfrac{\prod_{B\in\pp}\,\prod_{\n\in\I_B}\,(\mu^B_{\n})^{n^B_{\n}}}{\prod_{B\in\pp}\,\prod_{\n\in\I_B}\,(\mu^B_{\n})^{r^B_{\n}+n^B_{\n}}}.
$$
Note that
$$
\tfrac{(\nu^A_{\m})^{r^A_{\m}+n^A_{\m}}}{(\nu^A_{\m})^{n^A_{\m}}}=(\nu^A_{\m}+n^A_{\m})^{r^A_{\m}}\qquad\mbox{and}\qquad \tfrac{(\mu^B_{\n})^{r^B_{\n}+n^B_{\n}}}{(\mu^B_{\n})^{n^B_{\n}}}=(\mu^B_{\n}+n^B_{\n})^{r^B_{\n}}.
$$
Consequently,
$$
E\left(\left.\prod_{\i\in \I}\,[{\bf p}(\i)]^{r_{\i}}\right|{\bf N}\right)=\tfrac{\prod_{A\in\qq}\,\prod_{\m\in\I_A}\,(\nu^A_{\m}+n^A_{\m})^{r^A_{\m}}}{\prod_{B\in\pp}\,\prod_{\n\in\I_B}\,(\mu^B_{\n}+n^B_{\n})^{r^B_{\n}}}
$$
and thus it follows from \eqref{epr} and the fact that the distribution is uniquely determined by moments that the posterior distribution of ${\bf p}$ given the counts ${\bf N}={\bf n}$ is $\cP$-Dirichlet with parameters updated by counts. We note that the  parameters: $\nu^A_{\m}+n^A_{\m}$, $A\in \qq$, and $\mu^B_{\n}+n^B_{\n}$, $B\in \pp$, of the posterior distribution of ${\bf p}$ satisfy the constraints of the type \eqref{constr}. Indeed, this is due to the facts that  these constraints are linear in the parameters, that the original parameters $\nu^A_{\m}$ and $\mu^B_{\n}$ satisfy such constraints by assumption and that such constraints are also trivially satisfied by the counts $n^A_{\m}$ and $n^B_{\m}$. This proves that the $\cP$-Dirichlet forms a conjugate family of distribution.

The directed strong hyper Markov property of the $\cP$-Dirichlet holds  true for every $\p\in \cP$ because of the independences (see Def. \ref{rep-dir})  contained in its construction.
\end{proof}

\subsection{Applications}
Even in a small example such as Example \ref{41cont}, we see that, by going from the hyper Dirichlet to the $\cP$-Dirichlet, we have substantially increased the number of parameters and that  the parameters $\nu^{345}_{\m},\;\m\in \I_{345}$ are no longer constrained to add up to $\mu$ as they would be in the hyper Dirichlet. We therefore have more flexibility for the choice of the  $(\nu^{345}_{\m},\;m\in \I_{345})$ than we would have in the hyper Dirichlet.

 Let us assume that a preliminary studies has determined that the set of conditional independences between these variables is represented by the decomposable graph $G$ in our example. We now want to find the DAG model that best fits the data. But let us assume also that we are told by an expert that vertices $3$ and $4$ must be parents of $5$. In our DAG model search, it then does not make sense to include DAG's that contain directed edges going from $5$ to $4$ or from $5$ to $3$. The hyper Dirichlet is a prior on the $(p(\i), \i\in \I)$ that includes all possible DAG and therefore the right thing to do, in that case, is to take the $\cP$-Dirichlet as a prior in order to exclude the possibilities of having arrow going from $5$ to $4$ or $3$. Since, as pointed out above, we have more flexibility in the choice of the $\nu^{345}_{\m},\;m\in \I_{345}$, we can choose them to put more prior weight on the edges $4\rightarrow 5$ and $3\rightarrow 5$ if we wish.

Such situations are discussed, for example, in Angelopoulos and Cussens (2008). The focus of that paper is structure learning and they advocate the use of independent Dirichlet priors on the ${\bf p}^{v|\p_v}_{m|\k}$, which is precisely what the $\cP$-Dirichlet does for a selected number of orders. There is, of course, the restriction that the underlying graph is decomposable but  searches can often be started in the space of decomposable graphical models. Situations where one would typically want to use the $\cP$-Dirichlet would be ones where conditional independence according to a decomposable graphs are imposed by an expert and where, moreover, we are told that a variable must be a parent of another or a variable must be a root node or a variable must be a  leaf node or any other situation where a pattern is imposed for some of the directed edges.

\section{Characterizations  by local and global independence}
\label{characterization}
\subsection{The $\cP$-Dirichlet and the hyper-Dirichlet}
We now briefly recall the definition of local and global independence
\begin{definition}
{\em Let ${\bf p}$ be a vector of random probabilities associated with the graph $G$. We say that local parameter independence holds for ${\bf p}$ with respect to a DAG with a parent function $\p$  if for any $v\in V$ the random vectors
$$
(\P_{\bf p}(X_v=l|X_{\p_v}=\k),\,l\in\I_v),\quad \k\in\I_{\p_v},
$$
are independent (non-degenerate) and we say that global parameter independence holds for ${\bf p}$ if the random vectors
$$
\left((\P_{\bf p}(X_v=l|X_{\p_v}=\k),\,l\in\I_v),\quad \k\in\I_{\p_v}\right),\quad v\in V
$$
are independent (non-degenerate).}
\end{definition}

It appears that for such families of DAGs local and global independence of parameters characterizes the $\cP$-Dirichlet distribution. Actually, we do not need to distinguish between the two properties. Therefore instead we combine them into one property of {\em parameter independence}.

\begin{definition}
{\em Let ${\bf p}$ be a vector of random probabilities associated with the graph $G$.  We say that  parameter independence holds for ${\bf p}$ with respect to a DAG with a parent function $\p$ if both local and global independence hold for ${\bf p}$ or equivalently if the random vectors
$$
(\P_{\bf p}(X_v=l|X_{\p_v}=\k),\,l\in\I_v),\quad \k\in\I_{\p_v},\quad v\in V
$$
are independent (non-degenerate).

Analogously,  we say that parameter independence holds for ${\bf p}$ with respect to a family $\cP$ of DAGs  if it holds for any DAG from $\cP$.}
\end{definition}

We immediately note that under the condition of parameter independence with respect to a DAG having a parent function $\p$ for any $d$-way table $\r=(r_{\i},\,\i\in\I)$ of non-negative integers we have
\bel{DAGexp}
\E\,\prod_{\i\in\I}\,\left[\P_{\bf p}({\bf X}=\i)\right]^{r_{\i}}=\prod_{v\in V}\,\prod_{\k\in I_{\p_v}}\,\E\,\prod_{l\in\I_v}\,\left[\P_{\bf p}(X_v=l|X_{\p_v}=\k)\right]^{r^{\q_v}_{\k,l}},
\ee
where
\begin{equation}\label{marginal}
r^{\q_v}_{\k,l}=\sum_{\i\in\I:\;\i_{\q_v}=(\k,l)}\,r_{\i}.
\end{equation}

This property of parameter independence was crucial for the characterization of the Dirichlet distribution (for a complete graph)  given in Geiger and Heckerman (1997). We will extend this characterization to the $\cP$-Dirichlet distribution for various families $\cP$ of DAGs including those yielding the hyper-Dirichlet law. Such families are described by what we call a separating property.

\begin{definition}\label{separdef}
{\em A family $\cP$ of DAGs with structure $G$ is called {\em separating} if
\bel{separ}
\forall\,v\in V\;\;\;\;\exists\,\p,\,\p'\in\cP\quad\mbox{such that}\quad \p_v\ne\p'_v,
\ee}
\end{definition}

In our main result we will show that for separating families of DAGs parameter independence characterizes the $\cP$-Dirichlet distribution.
\begin{theorem}\label{main1} Let ${\bf p}$ be a vector of random probabilities associated with the graph $G$.
Let $\cP$ be a separating family of moral DAGs for a decomposable graph $G=(V,E)$.

If parameter independence for ${\bf p}$ with respect to $\cP$ holds then ${\bf p}$ has a $\cP$-Dirichlet distribution.
\end{theorem}

The proof is given in the Appendix.

The hyper-Dirichlet distribution is similarly characterized by parameter independence with respect to a $\cP$ which is rich enough. We state this result more precisely in the theorem below which is  an immediate consequence of Theorems \ref{main1} and \ref{hyp-dir}.

\begin{theorem}\label{main2} Let ${\bf p}$ be a vector of random probabilities associated with a decomposable graph $G$.
Let $\cP$, a separating family of moral DAGs with skeleton $G$,  satisfies \eqref{QQQ} and \eqref{sos}.

If parameter independence for ${\bf p}$ with respect to $\cP$ holds then ${\bf p}$ has a hyper-Dirichlet distribution.
\end{theorem}

\subsection{Special cases}
\subsubsection{The chain and the hyper Dirichlet distribution}
Let $G=(V,E)$ be a chain with vertices $V=\{1,\ldots,d\}$ and edges $E=\{\{i,i+1\},\,i=1,\ldots,d-1\}$. Then  $\cc=E$ and $\mathcal{S}=\{\{2\},\ldots,\{d-1\}\}$.

Consider the family $\cP=\{\p,\,\p'\}$, where
$$\p_1=\emptyset\quad\mbox{and}\quad\p_i=i-1,\quad i=2,\ldots,d,$$
and
$$\p'_d=\emptyset\quad\mbox{and}\quad\p'_i=i+1,\quad i=1,\ldots,d-1.$$
Note that $\cP$ is separating. Note also that $$\mathfrak{R}_{\p}=\p(V)\setminus\mathcal{S}=\{\{1\}\}\qquad\mbox{and}\qquad \mathfrak{R}_{\p'}=\p'(V)\setminus\mathcal{S}=\{\{d\}\}.$$ Therefore $\mathfrak{R}=\emptyset$ and  \eqref{QQQ} holds. Moreover, condition \eqref{sos} is clearly satisfied.
Thus from Theorem \ref{main2} we conclude the following result.
\begin{corollary}
Assume that the random vectors
\bel{1-n}
(\P_{\bf p}(X_j=l|X_{j-1}=k),\,l\in \I_j),\quad k\in \I_{j-1},\quad j=1,2\ldots,d
\ee
are jointly independent.

Assume also that the random vectors
\bel{n-1}
(\P_{\bf p}(X_j=l|X_{j+1}=m),\,l\in \I_j),\quad m\in \I_{j+1},\quad j=1,2\ldots,d
\ee
are jointly independent.

Then ${\bf p}$ has a hyper Dirichlet distribution with respect to $G$.

\end{corollary}
In the assumptions above we used the convention that $X_0=X_{d+1}=0$ and $\I_0=\I_{d+1}=\{0\}$.

Note that the family $\cP$ that we defined for the chain is the unique minimal separating family. That is any other family of DAGs is either non-separating or it contains $\cP$ as a proper subset.

Note also that for the two-chain (that is when $d=2$) we obtain the characterization of the classical Dirichlet distribution given in Th. 2 of Bobecka and Weso\l owski (2009). At the same time we extend the characterization given in Geiger and Heckerman (1997), Th. 2, where additionally to parameter independences as above it was assumed that densities exist and are sufficiently regular. Some of the regularity assumptions were considerably weakened in J\'arai (1998). More recently the entire Ch. 23 of the monograph J\'arai (2005) was devoted to this issue.
\subsubsection{The tree and the hyper Dirichlet distribution}
Let $G=(V,E)$ be a tree. As in  the case of the chain the set of cliques $\cc$ is equal to $E$ and  $\mathcal{S}=\{\{v\}:\,v\in V\setminus L\}$, where $L\subset V$ is the set of leaves, that is, those vertices which belong to exactly one edge. Any DAG  can be uniquely defined by  choosing a vertex $v$ such that $\p_v=\emptyset$. We will denote this DAG by $\mathcal{G}_v$.  Note that for any such  $\mathcal{G}_v$, for any $w\in V\setminus\{v\}$, the set $\p_w$  contains exactly one element .

Consider the family  $\cP=\{\mathcal{G}_v,\,v\in L\}$. Note that each vertex on the unique chain connecting $v$ and $w$ in $L$ has different parents in $\mathcal{G}_v$ and $\mathcal{G}_w$ and therefore $\cP$ is a separating family. Since any separator consists of only one vertex and any clique of only two vertices condition \eqref{sos} follows from the same observation. Since $\p_v(V)=\mathcal{S}\cup\{v\}$ it follows that $\mathfrak{R}_{\p_v}=\{\{v\}\}$. Consequently, \eqref{QQQ} is satisfied.
From Theorem \ref{main2} we have the following result for trees.

\begin{corollary}
Assume that for every leaf $v\in L$ the condition of parameter independence with respect to $\mathcal{G}_v$ holds.

Then $(\P_{\bf p}({\bf X}=\i),\,\i\in\I)$ follows the hyper Dirichlet distribution with respect to $G$.
\end{corollary}

\subsubsection{The complete graph and the classical Dirichlet distribution}
Consider a complete graph $G=(V,E)$, for which we, of course, have $\mathcal{C}=\{V\}$ and $\mathcal{S}=\emptyset$. Consider two   DAGs with parent functions $\p$ and $\p'$. Let $v_1,\ldots,v_d$ be the numbering corresponding to $\p$ and $v'_1,\ldots,v'_d$ that corresponding to $\p'$. Let us assume moreover that
\bel{pp'}
\forall\,\,j=1,\ldots,d\quad \p_{v_j}\ne \p'_{v'_j}.
\ee
We claim that the family $\cP=\{\p,\,\p'\}$ is separating. Since the cardinality of $\p_{v_j}$ is equal to $j-1$, $\p_{v_j}=\p_{v'_k}$ implies $j=k$. But this is forbidden for $j=1,\ldots,d-1$ by condition \eqref{pp'}. Therefore the family is separating. Similarly, $\mathfrak{R}_{\p}=\p(V)=\{\p_{v_j},\,j=2,\ldots, d\}$ and $\mathfrak{R}_{\p'}=\p'(V)=\{\p'_{v'_j},\,j=2,\ldots,d\}$ cannot have a common element by the cardinality argument combined with \eqref{pp'}. Thus \eqref{QQQ} is satisfied.  We therefore have the following result.
\begin{corollary}\label{GHCCC}
If for a complete graph $G=(V,E)$ parameter independence holds for any two DAGs satisfying \eqref{pp'} then  $(\P_{\bf p}(\underline{X}=\i),\,\i\in\I)$ has a classical Dirichlet distribution.
\end{corollary}

For $d=2$ this is the case of the chain discussed earlier in this section. In Heckerman, Geiger and Chickering (1995), Th. 7 (see also Geiger and Heckerman (1997), Th.3) the authors considered parameter independence for a  very special family of two DAGs with parent functions $\p$ and $\p'$ defined as follows
$$\p_1=\emptyset,\qquad \p_i=\{1,\ldots,i-1\},\quad i=2,\ldots,d,$$
and
$$\p'_i=\{d\}\cup\p_i,\quad i=1,\ldots,d-1,\qquad \p'_d=\emptyset.$$

\noindent Such a choice of family of DAGs was important for their proof in the case $d>2$, since then they could easily reduce the problem to the case $d=2$. As mentioned earlier, this case  was settled in Geiger and Heckerman (1997) under some smoothness assumptions on the density. Clearly, $\p$ and $\p'$ above satisfy \eqref{pp'} and their result follows by Corollary \ref{GHCCC} without assuming that the density exists.

Bobecka and Weso\l owski (2009) have exactly the same result as Corollary \ref{GHCCC} for $d=2$ only. Since they were interested in, so called, neutralities with respect to partitions, for their characterization of the Dirichlet for $d>2$, it was natural to assume parameter independences for the parameters of the marginal probability of one variable $X_l$ and those of the conditional probabilities of ${\bf X}_{V\setminus\{l\}}$ given $X_l$. They assumed that for any $l=1,\ldots,d$,
$$(\P_{\bf p}(X_l=i),\;i\in\I_l),\quad (\P_{\bf p}({\bf X}_{V\setminus\{l\}}=i_{V\setminus\{l\}}|X_l=i),\,\i_{V\setminus\{l\}}\in\I_{V\setminus\{l\}}),\quad i\in\I_l,$$
are independent. The approach was via moments and no density assumptions were required. A related result based on purely Bayesian argument was obtained in Ramamoorthi and Sangalli (2007). For related characterizations of the classical Dirichlet, see , e.g., Darroch and Ratcliff (1971), Fabius (1973), James and Mosimann (1980), Bobecka and Weso\l owski (2007), Chang et al. (2010), Sakowicz and Weso\l owski (2014) and the monograph Ng et al. (2011), Ch. 2.6.

\section{Conclusion}
\label{apps}
This paper makes two  contributions. The first one is practical and the second theoretical.
Let us consider a given set of variables and a given set of conditional independences between these variables which can be represented graphically by means of a moral DAG with skeleton $G$, decomposable. Assume that $\cP$ is the collection of a certain number of these moral DAG's and that we want to perform model selection in that class of DAG's.
Such a situation may arise when we are given a set of conditional independences and we are also given the direction of certain edges. In that case, we want to put a zero prior probability on any DAG that does not follow these requirements. Our new $\cP$-Dirichlet does just that and it has the advantage of increased flexibility in the choice of hyper parameters (due to the restricted number of DAG's in $\cP$). As we showed, the $\cP$-Dirichlet forms a conjugate family of prior distribution with the strong hyper Markov property.

Our theoretical contribution is a characterization based on local and global parameter independence of this new family of distributions and in particular of the hyper Dirichlet and the classical Dirichlet without the assumption of the existence of the density.

We would also like to emphasize that in the development of our new prior distribution, we have introduced new objects such as the sets $\pp$ and $\qq$ generalizing the notion of cliques and separators in a decomposable graph $G$. We have also shed light on the choice of two DAG's on a complete subset used by Geiger and Heckerman (1997) in their characterization by emphasizing that these two DAG's form a separating family  when $G$ is complete and that this particular choice of two DAG's is only one of many possible choices.

\vspace{5mm}\noindent
{\bf Acknowledgement.} H. Massam gratefully acknowledges support from NSERC Discovery Grant No A8947. J. Weso\l owski was partially supported by NCN grant 2012/05/B/ST1/00554.

\vspace{10mm}
\small
{\bf References}
\begin{enumerate}
\item
{\sc Angelopoulos, N. and  Cussens, J.}, Bayesian learning of
Bayesian networks with informative priors. {\em Annals
of Mathematics and Artificial Intelligence}, {\bf 54(1-3)} (2008),
53-98.

\item
{\sc Andersson, S. A., Madigan, D., Perlman, M. D. and Triggs, C. M. }, On the relation between conditional independence models determined by finite distributive lattices and directed acyclic graphs. {\em J. Statist. Plann. Infer.} {\bf 48} (1995), 25-46.

\item
{\sc Bobecka, K., Weso\l owski, J.}, The Dirichlet distribution and process through neutralities. {\em J. Theor. Probab.} {\bf 20} (2007), 295-308.

\item
{\sc Bobecka, K., Weso\l owski, J.}, Moments approach to characterizations of Dirichlet tables
through neutralities. {\em Publ. Math. Debrecen}  {\bf 74/3-4} (2009), 321-339.

\item {\sc Chang, W.Y., Gupta, R.D.,  Richards, D.St.P.},  Structural properties of the generalized Dirichlet distributions. In: {\em Algebraic Methods in Statistics and Probability, II} (M.A.G. Viana, H.P. Wynn, eds), {\em Contemp. Math.} {\bf 516} (2010), 109-124.

\item
{\sc Darroch, J., Ratcliff, D.}, A characterization of the Dirichlet distribution. {\em J. Amer. Statist. Assoc.} {\bf 66} (1971), 641-643.

\item
{\sc Dawid, P., Lauritzen, S.L.}, Hyper Markov laws in the statistical analysis of decomposable graphical models. {\em Ann. Statist.} {\bf 21} (1993), 1275-1317.

\item
{\sc Diaconis, P., Ylvisaker, D.}, Conjugate priors for exponential families. {\em Ann. Statist.} {\bf 7} (1979), 269-281.

\item
{\sc Fabius, J.}, Two characterizations of the Dirichlet distribution. {\em Ann. Statist.} {\bf 1} (1973), 583-587.

\item
{\sc Geiger, D., Heckerman, D.}, A characterization of the
Dirichlet distribution through global and local parameter
independence. {\em  Ann. Statist.} {\bf 25} (1997), 1344-1369.

\item
{\sc Heckerman, D. Geiger, D., Chickering, D.M.}, Learning
Bayesian networks: the combination of knowledge and statistical
data. {\em Machine Learn.} {\bf 20} (1995),  197-243.

\item
{\sc James, I.R., Mosimann, J.E.}, A new characterization of the Dirichlet distribution through neutrality. {\em Ann. Statist.} {\bf 8(1)} (1980), 183-189.

\item
{\sc Lauritzen, S.L.}, {\em Graphical Models}. Oxford Univ. Press, Oxford, 1996.

\item
{\sc J\'arai, A.}, Regularity property of the functional equation of the Dirichlet distribution.
{\em Aeq. Math.} {\bf 56} (1998), 37-46.

\item
{\sc  J\'arai, A.}, {\em Regularity Properties of Functional Equations in Several Variables}. Springer,
New York, 2005.

\item
{\sc Ng, K.W., Tian, G.-L., Tang, M.-L.}, {\em Dirichlet and Related Distributions, Theory, Methods and Applications}, Wiley, New York, 2011.

\item
{\sc Ramamoorthi, R.V., Sangalli, L.M.}, On a characterization of the Dirichlet distribution.
In: {\em Bayesian Statistics and Its Applications, (S. K. Upadhyay, U. Singh, D. K. Dey, eds.)},
Anamaya Publishers, New Delhi, 2007, 385-397.

\item
{\sc Sakowicz, A., Weso\l owski, J.}, Dirichlet distribution through neutralities with respect to two partitions. {\em J. Multivar. Anal.} {\bf 129} (2014), 1-15.

\item
{\sc Spiegelhalter, D., Lauritzen, S.}, Sequential updating of conditional probabilities on directed graphical structures.  {\em Networks}, {\bf 20}, 579-605, 1990.
\end{enumerate}

\normalsize
\vspace{5mm}
\section{Appendix}
\label{proofs}
\subsection{Proof of Theorem \ref{momenty}}
\begin{proof}
Assume that $({\bf p}(\i),\,i\in\I)$ is $\cP$-Dirichlet distributed. That is   representation \eqref{rep} holds and for any $\p\in\cP$, the random vectors   $({\bf p}^{v|\p_v}_{m|\k})_{m\in\I_v}$ follow a Dirichlet distribution $\mathrm{Dir}(\alpha^{v|\p_v}_{m|\k},\,m\in\I_v)$,  $\k\in\I_{\p_v}$, $v\in V$  and are independent.
That is, we have to identify the parameters $\alpha^{v|\p_v}_{m|\k}$, $m\in\I_v$, $\k\in\I_{\p_v}$, $v\in V$, such that (see \eqref{p-moments} and \eqref{summ})
\bel{dirmom}
\E\,\prod_{v\in V}\prod_{\k\in\I_{\p_v}}\,\prod_{m\in\I_v}\,\left({\bf p}^{v|\p_v}_{m|\k}\right)^{r^{\q_v}_{(\k,m)}}=\prod_{v\in V}\prod_{\k\in\I_{\p_v}}\,\tfrac{\prod_{m\in\I_v}\,(\alpha^{v|\p_v}_{m|\k})^{r^{\q_v}_{(\k,m)}}}{(\tilde{\alpha}^{\p_v}_{\k})^{r^{\p_v}_{\k}}},
\ee
where
\bel{diri}
\tilde{\alpha}^{\p_v}_{\k}=\sum_{m\in\I_v}\,\alpha^{v|\p_v}_{m|\k}\qquad\forall\,\k\in\I_{\p_v},
\ee
is equal to the right-hand side of \eqref{epr} with consistency conditions given in \eqref{constr}.

Note that the result in the opposite direction, that is the fact that the formula \eqref{epr} for moments, together with the constraints \eqref{constr}, implies $\cP$-Dirichlet distribution for $({\bf p}(\i),\,i\in\I)$, follows then immediately from the property that the distribution is uniquely determined by moments.

Consider an arbitrary collection $O_{\cP}$ of $\p$-perfect orders, $\p\in\cP$. Fix an arbitrary $\p\in\cP$ and consider a $\p$-perfect order $o\in O_{\cP}$. We will now relate the sets $Q^C_i$, $i=0,1,\ldots,j_C$, $C\in\cc$, to the numbering of vertices imposed by $o$ as given in \eqref{numb}. Clearly $C=C_l$ for some $l\in \{1,\ldots,K\}$. For ease of notation, we suppress the subscript $l$ in the remainder of this proof.  We define $j(l,i),\; i\in\{1,\ldots,j_C-1\}$ to be the index of the vertex  which satisfies
\bel{realQ}
Q^C_i=\q_{v_{l,j(l,i)}}=\p_{v_{l,j(l,i)+1}}\;.
\ee
We also note that
\bel{rest}
C=Q_0^C=\q_{v_l,c_l}\qquad\mbox{and}\qquad S=S_l=Q^C_{j_C}=\p_{v_{l,s_l+1}}.
\ee

We will now define the $\alpha^{v|\p_v}_{m|\k}$'s in terms of the $\nu^{A}$'s. For any $l\in\{1,\ldots,K\}$,  if $v=v_{l,j(l,i)}$,  set
\bel{d1}
\alpha^{v|\p_v}_{m|\k}:=\nu^{Q^C_i}_{(\k,m)}\qquad\forall\,(\k,m)\in\I_{Q_i^C}.
\ee
For any $l\in\{1,\ldots,K\}$,  if $v=v_{l,j}$ and $j\not =j(l,i)$ for any $i$,   we define \bel{ij}i_j=\min\{i:\,\q_v\subset Q^{C_l}_i\}\ee  and then set
\bel{d2}
\alpha^{v|\p_v}_{m|\k}:=\sum_{\n\in\I_{Q^C_{i_j}\setminus\q_v}}\,\nu_{(\n,m,k)}^{Q^C_{i_j}}\qquad \forall\,(\k,m)\in \I_{\q_v}.
\ee

We will now show that, if $v_{l,j}$ is such that $j\not =j(l,i)$ for any $i\in\{0,\ldots,j_C-1\}$, then
\bel{go}
\alpha^{v_{l,j}|\p_{v_{l,j}}}_{m|\k}=\tilde{\alpha}^{\p_{v_{l,j+1}}}_{(\k,m)}\quad\forall\,(\k,m)\in\I_{\p_{v_{l,j+1}}}.
\ee
Consider first the case  when $j(l,i_j)=j+1$. By  \eqref{d2} we have
$$
\alpha^{v_{l,j}|\p_{v_{l,j}}}_{m|\k}=\sum_{n\in\I_{v_{l,j+1}}}\,\nu_{(n,\k,m)}^{Q^C_{i_j}}
$$
Since $Q^C_{i_j}=\q_{v_{l,j+1}}$ from the above equality and \eqref{d1} we get
$$
\alpha^{v_{l,j}|\p_{v_{l,j}}}_{m|\k}=\sum_{\n\in\I_{v_{l,j+1}}}\,\alpha^{v_{l,j+1}|\p_{v_{l,j+1}}}_{n|(\k,m)}.
$$
Thus \eqref{go} follows from \eqref{diri}.

Second, consider the case $j(l,i_j)>j+1$. By \eqref{d2} we have
$$
\alpha^{v_{l,j}|\p_{v_{l,j}}}_{m|\k}=\sum_{\n\in\I_{Q^C_{i_j}\setminus \q_{v_{l,j}}}}\,\nu_{(\n,\k,m)}^{Q^C_{i_j}}=\sum_{n_1\in\I_{v_{l,j+1}}}\,
\sum_{\n_2\in\I_{Q^C_{i_j}\setminus\q_{v_{l,j+1}}}}\,\nu_{(n_1,\n_2,\k,m)}^{Q^C_{i_j}},
$$
where the second equality follows form the fact that $\q_{v_{l,j+1}}=\q_{v_{l,j}}\cup\{v_{l,j+1}\}$. Applying \eqref{d2} to the inner sum we obtain
$$
\alpha^{v_{l,j}|\p_{v_{l,j}}}_{m|\k}=\sum_{n_1\in\I_{v_{l,j+1}}}\,\alpha^{v_{l,j+1}|\p_{v_{l,j+1}}}_{n_1|(\k,m)}.
$$
Thus \eqref{go}  follows from \eqref{diri}.

Due to \eqref{go} we have cancelations in the right-hand side of \eqref{dirmom} and the only terms left are:
\begin{itemize}
\item in the numerator: $\alpha^{v|\p_v}_{m|\k}=\nu_{(\k,m)}^{Q^C_i}$ for $v=v_{l,j(l,i)}$, where $i\in\{0,\ldots,j_C-1\}$.

\item in the denominator: $\tilde{\alpha}^{\p_v}_{\k}$ for $v=v_{l,j(l,i)+1}$ where $i\in\{1,\ldots,j_C\}$.
\end{itemize}
In particular,  $j(l,j_C)=s_l$ in general and for $l=1$, $s_l=0$ so that, in the denominator, we have parameters indexed by $\p_{v_{l,s_l+1}}=S_l$ and $\p_{v_{1,1}}=\emptyset$.

To complete the proof, that is to show that the right-hand side of \eqref{dirmom} is equal to the right-hand side of \eqref{epr}, it remains to show that for any $l\in\{1,\ldots,K\}$ and $\i\in\{1,\ldots,j_C\}$,  we have
\bel{mius}
\mu^{Q^C_i}_{\k}=\tilde{\alpha}^{\p_{v_{l,j(l,i)+1}}}_{\k}\quad \forall\,\k\in\I_{Q^C_i}.
\ee

Note that \begin{itemize} \item[(i)] either $j(l,i)+1=j(l,i-1)$, that is $\q_{v_{l,j(l,i)+1}}=Q^C_{i-1}$, \item[(ii)] or $j(l,i)+1$ is not of the form $j(l,\tilde{i})$ for some $\tilde{i}\in\{i+1,\ldots,j_C\}$ (observe - see \eqref{ij} - that in this case we have $i_{j(l,i)+1}=i-1$). \end{itemize}

Let's consider  case (i) first. Using \eqref{diri} and then \eqref{d1} for all $\k\in\I_{Q^C_i}$ (note that $Q^C_i=\p_{v_{l,j(l,i)+1}}$) we obtain
$$
\tilde{\alpha}^{\p_{v_{l,j(l,i)+1}}}_{\k}=\sum_{m\in\I_{v_{l,j(l,i)+1}}}\,\alpha^{v_{l,j(l,i)+1}|\p_{v_{l,j+1}}}_{m|\k}
=\sum_{m\in\I_{v_{l,j(l,i)+1}}}\,\nu^{Q^C_{i-1}}_{(\k,m)}.
$$
Since $\{v_{l,j(l,i)+1}\}=\q_{v_{l,j(l,i)+1}}\setminus\q_{v_{l,j(l,j)}}=Q^C_{i-1}\setminus Q^C_i$, due to \eqref{constr} we obtain \eqref{mius}.

For case (ii),  we use again \eqref{diri} and then \eqref{d2} to arrive at
$$
\tilde{\alpha}^{\p_{v_{l,j(l,i)+1}}}_{\k}=\sum_{m\in\I_{v_{l,j(l,i)+1}}}\,\sum_{\n\in \I_{Q^C_{i-1}\setminus \q_{v_{l,j(l,i)+1}}}}\,\nu^{Q^C_{i-1}}_{(m,\n,\k)}=\sum_{(m,\n)\in\I_{Q^C_{i-1}\setminus \q_{v_{l,j(l,i)}}}}\,\nu^{Q^C_{i-1}}_{(m,\n,\k)},
$$
where the last equation follows from the fact that $\q_{v_{l,j(i,l)+1}}=\q_{v_{l,j(l,i)}}\cup\{v_{l,j(l,i)+1}\}$. Moreover $\q_{v_{l,j(l,i)}}=Q^C_i$, therefore \eqref{mius} follows now from  assumption \eqref{constr}.
\end{proof}
\subsection{Proof of Theorem \ref{main1}}
\begin{proof}
Note that according to Def. \ref{def02} it suffices to show that the formula for moments as given in \eqref{dirmom} and \eqref{diri} holds.

From \eqref{DAGexp} it follows that for any DAG from $\cP$ with parent function $\p$ and any $\r$
\bel{rmom}
\E\,\prod_{\i\in\I}\,\left[\P_{\bf p}({\bf X}=\i)\right]^{r_{\i}}=\prod_{v\in V}\,\prod_{\k\in \I_{\p_v}}\,f^{\p_v}_{\k}\left(r^{\q_v}_{\k,l},\,l\in \I_v\right),
\ee
where for any $v\in V$ and any $\k\in\I_{\p_v}$
$$
f^{\p_v}_{\k}(z_l,\,l\in\I_v)=\E\,\prod_{l\in\I_v}\,\left[\P_{\bf p}(X_v=l|{\bf X}_{\p_v}=\k)\right]^{z_l},\qquad z_l\in\{0,1,\ldots\},\; l\in \I_v.
$$

In order to prove Theorem \ref{main1} we  identify the functions $f^{\p_v}_{\k}$, $\k\in\I_{\p_v}$, $v\in V$. Our method relies on identification of the general form of the functions $f^{\p_v}_{\k}$, which will appears to be a ratio of products of gamma functions as in the formula for the moments of the $\cP$-Dirichlet distribution. Our main tool  is equation \eqref{tozs} of Lemma \ref{gammas}. The proof is divided into two parts, {\bf (a)} and {\bf (b)}. In part {\bf (a)} we transform the moment equation \eqref{rmom} into a the seemingly cumbersome but useful \eqref{crazy}. In part {\bf (b)} through a judicious choice of sparse $\r$'s in \eqref{crazy} we will obtain the general form of $f^{\p_v}_{\k}$'s using Lemma \ref{gammas}.

{\bf (a)} We first aim for the simplified functional equation \eqref{crazy}.
Fix an arbitrary $\underline{\tau}=(\tau_v\in\I_v,\,v\in V)$ and consider an $d$-way table $\ep=(\eps_{\i},\,\i\in\I)$ such that
$$
\eps_{\i}=\left\{\begin{array}{ll} 1, & \mbox{if}\;\;\i=\underline{\tau},\\
                                   0, & \mbox{otherwise}.\end{array}\right.
$$
Changing $\r$ into $\r+\ep$ in \eqref{rmom} we get
\begin{eqnarray} \label{fir}
\E\,\prod_{\i\in\I}\,\left[\P_{\bf p}({\bf X}=\i)\right]^{r_{\i}+\eps_{\i}}&=&\prod_{v\in V}\,[f^{\p_v}_{\tau_{\p_v}}\left(r^{\q_v}_{\tau_{\q_v}}+1,\;r^{\q_v}_{\tau_{\p_v},l},\,l\in\I_v\setminus\{\tau_v\}\right)\,  \\
&& \times\prod_{\k\in\I_{\p_v}\setminus\{\tau_{\p_v}\}}\,f^{\p_v}_{\k}\left(r^{\q_v}_{\k,l},\,l\in\I_v\right)]\nonumber
\end{eqnarray}

We will now obtain an equation of the type \eqref{tozs} by equating the right-hand side of \eqref{fir} for different $\p$'s from $\cP$.
Fix a DAG in $\cP$, that is a $\p\in\cP$, and fix a vertex $v\in V$. Then, by separation property \eqref{separ} there exists another DAG in $\cP$ with parent function $\p'$ such that $\p'_v\ne \p_v$. For each of $\p$ and $\p'$ the right-hand side of \eqref{fir} is split into three parts: the first (first line) concerns $v$, the second (second line) $\c_v$ and the third (third line) the remainder of $V$. Thus we obtain
\begin{eqnarray}\label{As}
&&\hspace{1.2cm}f^{\p_v}_{\tau_{\p_v}}\left(r^{\q_v}_{\tau_{\q_v}}+1,\;r^{\q_v}_{\tau_{\p_v},l},\,l\in\I_v\setminus\{\tau_v\}\right)\,
\prod_{\k\in\I_{\p_v}\setminus\{\tau_{\p_v}\}}\,f^{\p_v}_{\k}\left(r^{\q_v}_{\k,l},\,l\in\I_v\right)\nonumber\\
&&\prod_{w\in \c_v}\,f^{\p_w}_{\tau_{\p_w}}\left(r^{\q_w}_{\tau_{\q_w}}+1,\,
r^{\q_w}_{\tau_{\p_w},l},\,l\in\I_w\setminus\{\tau_w\}\right)\,
\prod_{\k\in\I_{\p_w}\setminus\{\tau_{\p_w}\}}\,f^{\p_w}_{\k}\left(r^{\q_w}_{\k,l},\,l\in\I_w\right)\nonumber\\
&&\prod_{w\not\in \c_v\cup\{v\}}\,f^{\p_w}_{\tau_{\p_w}}\left(r^{\q_w}_{\tau_{\q_w}}+1,\,
r^{\q_w}_{\tau_{\p_w},l},\,l\in\I_w\setminus\{\tau_w\}\right)\,
\prod_{\k\in\I_{\p_w}\setminus\{\tau_{\p_w}\}}\,f^{\p_w}_{\k}\left(r^{\q_w}_{\k,l},\,l\in\I_w\right)
\nonumber\\
&&\hspace{1cm} =
f^{\p'_v}_{\tau_{\p'_v}}\left(r^{\q'_v}_{\tau_{\q'_v}}+1,\;r^{\q'_v}_{\tau_{\p'_v},l},\,l\in\I_v\setminus\{\tau_v\}\right)\,
\prod_{\k\in\I_{\p'_v}\setminus\{\tau_{\p'_v}\}}\,f^{\p'_v}_{\k}\left(r^{\q'_v}_{\k,l},\,l\in\I_v\right)\\
&&\prod_{w\in \c'_v}\,f^{\p'_w}_{\tau_{\p'_w}}\left(r^{\q'_w}_{\tau_{\q'_w}}+1,\,
r^{\q'_w}_{\tau_{\p'_w},l},\,l\in\I_w\setminus\{\tau_w\}\right)\,
\prod_{\k\in\I_{\p'_w}\setminus\{\tau_{\p'_w}\}}\,f^{\p'_w}_{\k}\left(r^{\q'_w}_{\k,l},\,l\in\I_w\right)\nonumber\\
&&\prod_{w\not\in \c'_v\cup\{v\}}\,f^{\p'_w}_{\tau_{\p'_w}}\left(r^{\q'_w}_{\tau_{\q'_w}}+1,\,
r^{\q'_w}_{\tau_{\p'_w},l},\,l\in\I_w\setminus\{\tau_w\}\right)\,
\prod_{\k\in\I_{\p'_w}\setminus\{\tau_{\p'_w}\}}\,f^{\p'_w}_{\k}\left(r^{\q'_w}_{\k,l},\,l\in\I_w\right)\nonumber
\end{eqnarray}

We now write this equation above for two distinct values first for $\tau_v=\rho$ and then for $\tau_v=\sigma$ in $\I_v$, while keeping $\tau_k$ the same for all $k\ne v$. We obtain two equations, say $E_{\rho}$ and $E_{\sigma}$ and we then write the identity
\begin{equation}\label{Es}
\frac{\mathrm{lhs}(E_{\rho})}{\mathrm{lhs}(E_{\sigma})}=\frac{\mathrm{rhs}(E_{\rho})}{\mathrm{rhs}(E_{\sigma})}.
\end{equation}

Many simplifications occur (see part \ref{cra} of Appendix) and we arrive at
{\tiny \begin{eqnarray}\label{crazy}
&&\hspace{1.2cm}\frac{f^{\p_v}_{\tau_{\p_v}}\left(r^{\q_v}_{\tau_{\p_v},\rho}+1,\;r^{\q_v}_{\tau_{\p_v},l},\,l\in\I_v\setminus\{\rho\}\right)}
{f^{\p_v}_{\tau_{\p_v}}\left(r^{\q_v}_{\tau_{\p_v},\sigma}+1,\;r^{\q_v}_{\tau_{\p_v},l},\,l\in\I_v\setminus\{\sigma\}\right)}\nonumber\\
&&\prod_{w\in \c(v)}\,\frac{f^{\p_w}_{(\tau_{\p_w\setminus\{v\}},\rho)}\left(r^{\q_w}_{(\tau_{\p_w\setminus\{v\}},\rho),\tau_w}+1,\,
r^{\q_w}_{(\tau_{\p_w\setminus\{v\}},\rho),l},\,l\in\I_w\setminus\{\tau_w\}\right)}
{f^{\p_w}_{(\tau_{\p_w\setminus\{v\}},\sigma)}\left(r^{\q_w}_{(\tau_{\p_w\setminus\{v\}},\sigma),\tau_w}+1,\,
r^{\q_w}_{(\tau_{\p_w\setminus\{v\}},\sigma),l},\,l\in\I_w\setminus\{\tau_w\}\right)}\,
\frac{f^{\p_w}_{(\tau_{\p_w\setminus\{v\}},\sigma)}\left(r^{\q_w}_{(\tau_{\p_w\setminus\{v\}},\sigma),l},\,l\in\I_w\right)}
{f^{\p_w}_{(\tau_{\p_w\setminus\{v\}},\rho)}\left(r^{\q_w}_{(\tau_{\p_w\setminus\{v\}},\rho),l},\,l\in\I_w\}\right)}\nonumber\\
&&\hspace{3cm}=\frac{f^{\p'_v}_{\tau_{\p'_v}}\left(r^{\q'_v}_{\tau_{\p'_v},\rho}+1,\;r^{\q'_v}_{\tau_{\p'_v},l},\,l\in\I_v\setminus\{\rho\}\right)}
{f^{\p'_v}_{\tau_{\p'_v}}\left(r^{\q'_v}_{\tau_{\p'_v},\sigma}+1,\;r^{\q'_v}_{\tau_{\p'_v},l},\,l\in\I_v\setminus\{\sigma\}\right)}\\
&&\prod_{w\in \c'(v)}\,\frac{f^{\p'_w}_{(\tau_{\p'_w\setminus\{v\}},\rho)}\left(r^{\q'_w}_{(\tau_{\p'_w\setminus\{v\}},\rho),\tau_w}+1,\,
r^{\q'_w}_{(\tau_{\p'_w\setminus\{v\}},\rho),l},\,l\in\I_w\setminus\{\tau_w\}\right)}
{f^{\p'_w}_{(\tau_{\p'_w\setminus\{v\}},\sigma)}\left(r^{\q'_w}_{(\tau_{\p'_w\setminus\{v\}},\sigma),\tau_w}+1,\,
r^{\q'_w}_{(\tau_{\p'_w\setminus\{v\}},\sigma),l},\,l\in\I_w\setminus\{\tau_w\}\right)}\,
\frac{f^{\p'_w}_{(\tau_{\p'_w\setminus\{v\}},\sigma)}\left(r^{\q'_w}_{(\tau_{\p'_w\setminus\{v\}},\sigma),l},\,l\in\I_w\right)}
{f^{\p'_w}_{(\tau_{\p'_w\setminus\{v\}},\rho)}\left(r^{\q'_w}_{(\tau_{\p'_w\setminus\{v\}},\rho),l},\,l\in\I_w\}\right)}\nonumber
\end{eqnarray}}

{\bf (b)} We now simplify \eqref{crazy} further by writing it for properly chosen sparse  $\r$'s. This will lead us to functional equations for functions defined on $\I_v$.

We define
$$\d_v=\p'_v\cap\c_v\qquad\mbox{and}\qquad\d'_v=\p_v\cap\c'_v.$$
Note that due to the separation property \eqref{separ} at least one of them is not empty. Without loss of generality let us assume that $\d_v\neq\emptyset$. Fix $\xi_{\d_v}\in\I_{\d_v}$ such that $\xi_i\ne\tau_i$ for any $i\in\d_v$. For any $l\in\I_v$ denote by
$\i(l)$ the cell with labels
$$i_v=l,\;i_{\d_v}=\xi_{\d_v},\;i_y=\tau_y\;\mbox{for}\;y\not\in\d_v\cup\{v\}.
$$

Define $$x_l=r_{\i(l)}\,\qquad l\in\I_v.$$

Consider any $\r=(r_{\i})$ such that $r_{\i}=0$ for all $\i\not\in\{\i(l),\,l\in\I_v\}$.
Since $\p_v\cap\d_v=\emptyset$, by \eqref{marginal}
\begin{equation}\label{xl}
r^{\q_v}_{\tau_{\p_v},l}=x_l,\qquad l\in\I_v.
\end{equation}
Again, by \eqref{marginal} and since $\p'_v\supset \d_v\ne\emptyset$ we have
\begin{equation}\label{r0}
r^{\q'_v}_{\tau_{\p'_v},l}=0,\qquad l\in \I_v.
\end{equation}
Moreover, for $l\in\I_w$ and $k\in\I_v$ (particularly for $k=\rho$ or $k=\sigma$, which we shall use here)
$$
r^{\q_w}_{(\tau_{\p_w\setminus\{v\}},k),l}=\left\{\begin{array}{ll} x_k, & \mbox{if}\;\;\p_w\cap\d_v=\emptyset\;\mbox{and either}\; (w\not\in \d_v,\;\mbox{and}\;l=\tau_w)\;\;\mbox{or}\;(w\in\d_v\;\;\mbox{and}\;\;l=\xi_w), \\
0, & \mbox{otherwise} \end{array}\right.
$$
and
$$
r^{\q'_w}_{(\tau_{\p'_w\setminus\{v\}},k),l}=\left\{\begin{array}{ll} x_k, & \mbox{if}\;\;\p'_w\cap\d_v=\emptyset\;\mbox{and either}\; (w\not\in \d_v,\;\mbox{and}\;l=\tau_w)\;\;\mbox{or}\;(w\in\d_v\;\;\mbox{and}\;\;l=\xi_w), \\
0, & \mbox{otherwise}. \end{array}\right.
$$

These last two observations imply that the products $\prod_{w\in\c_v}$ and $\prod_{w\in\c'_v}$ in \eqref{crazy}  factor into a function of $x_\rho$ and a function of $x_{\sigma}$. Therefore their quotient can be written as $a_{v,\rho}(x_{\rho})/a_{v,\sigma}(x_{\sigma})$. Note, that potentially these functions may depend of $\p$ and $\p'$, but it will not impact our final result.

Moreover, by  \eqref{xl} and \eqref{r0} it follows that \eqref{crazy} assumes the form
$$
\frac{f^{\p_v}_{\tau_{\p_v}}\left(x_{\rho}+1,\;x_l,\,l\in\I_v\setminus\{\rho\}\right)}
{f^{\p_v}_{\tau_{\p_v}}\left(x_{\sigma}+1,\;x_l,\,l\in\I_v\setminus\{\sigma\}\right)}=K_v\,\frac{a_{v,\rho}(x_{\rho})}{a_{v,\sigma}(x_{\sigma})},
$$
where
$$
K_v=\frac{f^{\p'_v}_{\tau_{\p'_v}}\left(1_{\rho},\;0_l,\,l\in\I_v\setminus\{\rho\}\right)}
{f^{\p'_v}_{\tau_{\p'_v}}\left(1_{\sigma},\;0_l,\,l\in\I_v\setminus\{\sigma\}\right)}.
$$

Since $\tau_{\p_v}$ was arbitrary in $\I_{\p_v}$ we conclude from Lemma \ref{gammas} that for any $\k\in\I_{\p_v}$ either
\bel{fpv}
f^{\p_v}_{\k}(z_l,\,l\in\I_v)=\frac{\prod_{l\in\I_v}\,\left(A^{v|\p_v}_{l|\k}\right)^{z_l}}{\left(\tilde{A}^{\p_v}_{\k}\right)^{|z|}},
\ee
where $\tilde{A}^{\p_v}_{\k}=\sum_{l\in\I_v}\,A^{v|\p_v}_{l,\k}$ (recall that $(A)^z=A(A+1)\ldots(A+z-1)$ is the ascending Pochhammer symbol)
or it is a product of univariate power functions
\bel{fpv1}
f^{\p_v}_{\k}(z_l,\,l\in\I_v)=\prod_{l\in\I_v}\,\left[A^{v|\p_v}_{l|\k}\right]^{z_l}.
\ee
However the latter case is impossible due to the parameter independence assumption which requires that the distribution of the random vector $(\P_{\bf p}(X_v=l|{\bf X}_{\p_v}=\k),\,l\in\I_v)$ is non-degenerate.

We now want to identify the functions $f^{\p'_v}_{\k}$, $\k\in \I_{\p'_v}$. If $\d'_v\ne\emptyset$ we can repeat the argument used to derive $f^{\p_v}_{\k}$ and obtain an analogue of \eqref{fpv} with $\p$ replaced by $\p'$. If $\d'_v=\emptyset$ we need to do some more work. We will use another sparse $\r$ with new $\i(l)$'s defined by substituting $\d'_v$ for $\d_v$. Note that under this new sparsity pattern for any $k\in\I_v$ (particularly for $k=\rho$ or $k=\sigma$, which we shall use here)
$$
r^{\q_w}_{(\tau_{\p_w\setminus\{v\}},k),l}=r^{\q'_w}_{(\tau_{\p'_w\setminus\{v\}},k),l}=\left\{\begin{array}{ll} x_k & \mbox{if}\;l=\tau_w, \\
                                                                     0 & \mbox{if}\;l\ne \tau_w, \end{array}\right.
$$
and
$$
r^{\q_v}_{\tau_{\p_v},l}=r^{\q'_v}_{\tau_{\p'_v},l}=x_l,\qquad l\in\I_v.
$$

Thus, \eqref{crazy} becomes
$$
\frac{f^{\p_v}_{\tau_{\p_v}}\left(x_{\rho}+1,\;x_l,\,l\in\I_v\setminus\{\rho\}\right)}
{f^{\p_v}_{\tau_{\p_v}}\left(x_{\sigma}+1,\;x_l,\,l\in\I_v\setminus\{\sigma\}\right)}=\frac{a_{v,\rho}(x_{\rho})}{a_{v,\sigma}(x_{\sigma})}\,
\frac{f^{\p'_v}_{\tau_{\p'_v}}\left(x_{\rho}+1,\;x_l,\,l\in\I_v\setminus\{\rho\}\right)}
{f^{\p'_v}_{\tau_{\p'_v}}\left(x_{\sigma}+1,\;x_l,\,l\in\I_v\setminus\{\sigma\}\right)}.
$$

Plugging \eqref{fpv} into the left hand side above we obtain
$$
\frac{f^{\p'_v}_{\tau_{\p'_v}}\left(x_{\rho}+1,\;x_l,\,l\in\I_v\setminus\{\rho\}\right)}
{f^{\p'_v}_{\tau_{\p'_v}}\left(x_{\sigma}+1,\;x_l,\,l\in\I_v\setminus\{\sigma\}\right)}=\,\frac{A^{v|\p_v}_{\rho|\tau_{\p_v}}+x_{\rho}}{A^{v|\p_v}_{\sigma|\tau_{\p_v}}+x_{\sigma}}\,
\frac{a'_{v,\sigma}(x_{\sigma})}{a'_{v,\rho}(x_{\rho})}\;\;\mbox{or}\;\;\,\frac{A^{v|\p_v}_{\rho|\tau_{\p_v}}}{A^{v|\p_v}_{\sigma|\tau_{\p_v}}}\,
\frac{a'_{v,\sigma}(x_{\sigma})}{a'_{v,\rho}(x_{\rho})},
$$
respectively. Again we use Lemma \ref{gammas} to conclude that one of the representations \eqref{fpv} or \eqref{fpv1} (with $\p$ changed into $\p'$) holds also for $f^{\p'_v}_{\k}$ for any $\k\in\I_{\p'_v}$. Similarly, as above, we conclude that non-degeneracy implies that the representation given in \eqref{fpv1} is not a valid one.

Given an arbitrary $v\in V$, so far,  we have derived  the expression  of $f^{\p_v}_{\k}$ and $f^{\p'_v}_{\k}$ in \eqref{fpv} for an arbitrary separating pair $\p,\p'\in\cP$. Clearly, \eqref{fpv} is valid for any $\p\in\cP$ and any $v\in V$. Indeed, given $v\in V$ and the separating pair $\p,\,\p'$, consider another $\p''$. Then, either $\p''(v)\ne\p(v)$ and then we consider the pair $\p,\,\p''$ or $\p''(v)\ne\p'(v)$ and then we consider the pair $\p',\,\p''$.

Now, returning to \eqref{rmom} we see that for any $d$-way table $\r=(r_{\i},\,\i\in\I)$ of non-negative numbers,  any $\p\in\cP$, any $v\in V$ and any  $(\k,l)\in\I_{\q_v}$  there exist numbers $A^{v|\p_v}_{l|\k}$  such that
\bel{repr}
\E\,\prod_{\i\in\I}\,\left[\P_{\bf p}({\bf X}=\i)\right]^{r_{\i}}=\prod_{v\in V}\,\prod_{\k\in I_{\p_v}}\,\frac{\prod_{l\in\I_v}\,\left(A^{v|\p_v}_{l|\k}\right)^{r^{\p_v,v}_{\k,l}}}{\left(\tilde{A}^{\p_v}_{\k}\right)^{r^{\p_v}_{\k}}},
\ee
where
\bel{adot}
\tilde{A}^{\p_v}_{\k}=\sum_{l\in\I_v}\,A^{v|\p_v}_{l|\k}.
\ee
\end{proof}

\subsection{An auxiliary result on a functional equation}
\begin{lemma}\label{gammas}
Let $F$ be a positive function defined on $n$th cartesian product of non-negative integers such that $F(\underline{0})=1$ and
\bel{sum1}
F(\x)=\sum_{i=1}^n\,F(\x+\ep_i),
\ee
where $\ep_i$ has all  components equal to 0 except $i$th component which is 1.
Assume that for any distinct $p,q\in\{1,\ldots,n\}$
\bel{tozs}
\frac{F(\x+\ep_p)}{F(\x+\ep_q)}=\frac{h_p(x_p)}{h_q(x_q)}\qquad\forall\,\x=(x_1,\ldots,x_n)\in\{0,1,\ldots\}^n
\ee
for some functions $h_i$, $i=1,2,\ldots,n$.

Then  there  exists a vector $\underline{A}=(A_1,\ldots, A_n)\in\R^n$ such that $\forall\,\x=(x_1,\ldots,x_n)\in\{0,1,\ldots\}^n$ either
$$
F(\x)=\frac{\prod_{i=1}^n\,(A_i)^{x_i}}{(|\underline{A}|)^{|\x|}},
$$
where $|\underline{u}|=u_1+\ldots+u_n$ for any vector $\underline{u}=(u_1,\ldots,u_n)$ and $(a)^m=a(a+1)\ldots (a+m-1)$,
or
$$
F(\x)=\prod_{i=1}^n\,x_i^{A_i}.
$$
\end{lemma}

Lemma \ref{gammas}, as given above, is a special version of Lemma 3.1 from Sakowicz and Weso\l owski (2014) (it suffices just to take $A+\{1,\ldots,n\}$ in this lemma). It is also closely related to the argument used in the proof of Theorem 2 in Bobecka and Weso\l owski (2009).

\subsection{Proof of \eqref{crazy}}\label{cra}
$\,$

\vspace{2mm}\noindent
{\bf 1.} Let $A_1(\rho)$ and $A_1(\sigma)$ be the values of $A_1=\prod_{\k\in\I_{\p_v}\setminus\{\tau_v\}}\,f_{\k}^{\p_v}(r_{\k,l}^{\q_v},\,l\in\I_v)$, the second factor in the first line of the left hand side of \eqref{As}, for $\tau_v=\rho$ and $\tau_v=\sigma$, respectively. Let $A_1'(\rho)$ and $A_1'(\sigma)$ be the analog quantities for the right-hand side of \eqref{As}.  Clearly
$$
A_1(\rho)=\prod_{\k\in\I_{\p_v}\setminus\{\tau_{\p_v}\}}\,f_{\k}^{\p_v}(r_{\k,\rho}^{\q_v},\,r_{\k,\sigma}^{\q_v},\,r_{\k,l}^{\q_v},\,l\in\I_v\setminus\{\rho,\sigma\})=A_1(\sigma).
$$
Similarly $A_1'(\rho)=A_1'(\sigma)$.

\vspace{2mm}\noindent
{\bf 2.} Let $A_2(\rho)$ and $A_2(\sigma)$ be the values of the factor for fixed $w\in\c_v$ in the second line of the left hand side of \eqref{As}, for $\tau_v=\rho$ and $\tau_v=\sigma$, respectively. Let $A_2'(\rho)$ and $A_2'(\sigma)$ be the analog quantities for the right-hand side of \eqref{As}.  Clearly
\begin{eqnarray*}
A_2(\rho)&=&f_{(\tau_{\p_w\setminus\{v\}},\rho)}^{\p_w}(r^{\q_w}_{(\tau_{\p_w\setminus\{v\}},\rho),\tau_w}+1,\,r^{\q_w}_{(\tau_{\p_w\setminus\{v\}},\rho),l},l\in\I_w\setminus\{\tau_w\})\\ &&\hspace{0.5cm}f_{(\tau_{\p_w\setminus\{v\}},\sigma)}^{\p_w}(r^{\q_w}_{(\tau_{\p_w\setminus\{v\}},\sigma),l},\,l\in\I_w)\,\prod_{\k\in\I_{\p_w}\setminus\{\tau_{\p_w},(\tau_{\p_w\setminus\{v\}},\sigma)\}}\,
f_{\k}^{\p_w}(r^{\q_w}_{\k,l},\,l\in\I_w)
\end{eqnarray*}
and
\begin{eqnarray*}
A_2(\sigma)&=&f_{(\tau_{\p_w\setminus\{v\}},\sigma)}^{\p_w}(r^{\q_w}_{(\tau_{\p_w\setminus\{v\}},\sigma),\tau_w}+1,\,r^{\q_w}_{(\tau_{\p_w\setminus\{v\}},\sigma),l},l\in\I_w\setminus\{\tau_w\})\\ &&\hspace{0.5cm}f_{(\tau_{\p_w\setminus\{v\}},\rho)}^{\p_w}(r^{\q_w}_{(\tau_{\p_w\setminus\{v\}},\rho),l},\,l\in\I_w)\,\prod_{\k\in\I_{\p_w}\setminus\{\tau_{\p_w},(\tau_{\p_w\setminus\{v\}},\rho)\}}\,
f_{\k}^{\p_w}(r^{\q_w}_{\k,l},\,l\in\I_w).
\end{eqnarray*}
Note that the two sets appearing in the indices: $\{\tau_{\p_w},(\tau_{\p_w\setminus\{v\}},\sigma)\}$ in $A_2(\rho)$ and $\{\tau_{\p_w},(\tau_{\p_w\setminus\{v\}},\rho)\}$ in $A_2(\sigma)$ are identical. Therefore, for each $w\in\c_v$ the ratio $\frac{A_2(\rho)}{A_2(\sigma)}$ is equal to the factor in the second line of the left hand side of \eqref{crazy}. Similarly, for each $w\in\c_v$ the ratio $\frac{A_2'(\rho)}{A_2'(\sigma)}$ is equal to the factor in the second line of the right-hand side of \eqref{crazy}.

\vspace{2mm}\noindent
{\bf 3.} Since $v$  appears  in the third line of neither the left hand side nor the right-hand side of \eqref{As}, these lines cancel out in the ratios of \eqref{Es}.

\end{document}